\documentclass[12pt,reqno]{amsart}

\makeatletter
\def\section{\@startsection{section}{1}%
  \z@{1.2\linespacing\@plus\linespacing}{.5\linespacing}%
  {\normalfont\scshape\centering}}
\def\subsection{\@startsection{subsection}{2}%
  \z@{0.9\linespacing\@plus.7\linespacing}{-.5em}%
  {\normalfont\bfseries}}
\makeatother

\usepackage{amsmath}
\usepackage{amsfonts}
\usepackage{amsbsy}
\usepackage{amssymb}
\usepackage[normalem]{ulem}
\usepackage{times}
\usepackage{mathrsfs}
\usepackage{exscale}
\usepackage{graphicx}
\usepackage[colorlinks,citecolor=blue,linkcolor=rouge,
            bookmarksopen,
            bookmarksnumbered
           ]{hyperref}

\usepackage{times}
\usepackage[nomessages]{fp}

\usepackage{color}
\definecolor{gr}{rgb}   {0.,   0.6,   0.25 }
\definecolor{mg}{rgb}   {0.85,  0.,    0.85}
\definecolor{marin}{rgb}   {0.,   0.65,   0.25}
\definecolor{rouge}{rgb}   {0.8,   0.,   0.}
\definecolor{orange}{rgb}   {0.8,   0.4,   0.}

\newcommand{\Bk}{\color{black}}

\newtheorem{theorem}{Theorem}[section]
\newtheorem{lemma}[theorem]{Lemma}

\theoremstyle{definition}

\theoremstyle{remark}
\newtheorem{remark}[theorem]{Remark}

\numberwithin{equation}{section}

\newcommand{\ee}{\hskip0.15ex}
\newcommand{\me}{\hskip-0.15ex}
\newcommand{\dd}[1]{_{\raise-0.3ex\hbox{$\scriptstyle #1$}}}

\newcommand\rd{{\mathrm d}}

\renewcommand{\div}{\operatorname{\rm div}}

\newcommand\R{{\mathbb R}}

\newcommand\T{\mathbb{T}}
\newcommand\Z{{\mathbb Z}}

\newcommand\cA{{\mathcal A}}
\newcommand\B{{\mathsf{B}}}

\newcommand\cC{{\mathcal C}}

\newcommand\F{{\mathscr F}}

\newcommand\cI{{\mathcal I}}

\newcommand\cO{{\mathcal O}}

\newcommand\cR{{\mathcal R}}
\newcommand\cS{{\mathcal S}}
\newcommand\sT{{\mathcal T}}

\newcommand{\bB}{\boldsymbol{\mathsf{B}}}

\newcommand{\be}{{\boldsymbol{\mathsf{E}}}}

\newcommand{\bK}{\boldsymbol{\mathsf{K}}}
\newcommand{\bL}{\boldsymbol{\mathsf{L}}}
\newcommand{\bM}{\boldsymbol{\mathsf{M}}}
\newcommand{\bn}{\boldsymbol{\mathsf{N}}}
\newcommand{\bp}{\boldsymbol{\mathsf{P}}}
\newcommand{\bu}{\boldsymbol{\mathsf{u}}}
\newcommand{\bv}{\boldsymbol{\mathsf{v}}}

\newcommand{\bt}{\boldsymbol{t}}

\newcommand{\rA}{{\mathsf{A}}}
\newcommand{\rb}{{\mathsf{b}}}
\newcommand{\rc}{{\mathsf{c}}}
\newcommand{\rB}{{\mathsf{B}}}

\newcommand{\rG}{{\mathsf{G}}}

\newcommand{\rh}{{\mathsf{g}}}
\newcommand{\rH}{{\mathsf{H}}}

\newcommand{\rL}{{\mathsf{L}}}
\newcommand{\ru}{{\mathsf{u}}}
\newcommand{\rv}{{\mathsf{v}}}

\newcommand{\La}{{\ee\bL\me}}
\newcommand{\tbu}{\bu^*}

\newcommand {\am}{\mathsf{a}}

\newcommand {\Id}{\mathbb I}

\newcommand{\meas}{\mathrm{meas}}

\usepackage{soul}

\begin{document}

\title[High frequency oscillations of first eigenmodes]{\bf High frequency oscillations\\ of first eigenmodes  in axisymmetric shells \\ as the thickness tends to zero  }

\author{Marie Chaussade-Beaudouin}

\author{Monique Dauge}

\author{Erwan Faou}

\author{Zohar Yosibash}







\begin{abstract}
The lowest eigenmode of thin axisymmetric shells is investigated for two physical models (acoustics and elasticity) as the shell thickness ($2\varepsilon$) tends to zero. Using a novel asymptotic expansion we determine the behavior of the eigenvalue $\lambda(\varepsilon)$ and the eigenvector
angular frequency $k(\varepsilon)$ for shells with Dirichlet boundary conditions along the lateral boundary, and  natural boundary conditions on the other parts.

First, the scalar Laplace operator for acoustics is addressed, for which $k(\varepsilon)$ is always zero. In contrast to it, for the Lam\'e system of linear elasticity several different  types of shells are defined, characterized by their geometry, for which $k(\varepsilon)$ tends to infinity as $\varepsilon$ tends to zero.
For two families of shells: cylinders and elliptical barrels we explicitly provide $\lambda(\varepsilon)$ and
$k(\varepsilon)$ and demonstrate by numerical examples the different behavior as $\varepsilon$ tends to zero.
\end{abstract}
\maketitle

{
\parskip 1pt
\tableofcontents
}


\section{Introduction}

The lowest natural frequency of shell-like structures is of major importance in engineering because it is one of the driving considerations in designing thin structures (for example containers). It is associated with linear isotropic elasticity, governed by the Lam\'e system.
The elastic lowest eigenmode in axisymmetric homogeneous isotropic shells was address by W.\ Soedel \cite{SoedelEncyclopedia} in the {\em Encyclopedia of Vibration}:
\begin{quote}
\emph{
[We observe] a phenomenon which is particular to many deep shells, namely that the lowest natural frequency does not correspond to the simplest natural mode, as is typically the case for rods, beams, and plates.}
\end{quote}
This citation emphasizes that for shells these lowest natural frequencies may hide some interesting ``strange'' behavior.
The expression ``deep shell'' contrasts with ``shallow shells'' for which the main curvatures are of same order as the thickness. Typical examples of deep shells are cylindrical shells, spherical caps, or barrels (curved cylinders).

In acoustics, driven by the scalar Laplace operator, it is well known that, when Dirichlet conditions are applied on the whole boundary, the first eigenmode is simple in both senses that it is not multiple and that it is invariant by rotation. We will revisit this result, in order to extend it to mixed Dirichlet-Neumann conditions. In contrast to the scalar Laplace operator, the simple behavior of the first eigenmode does not carry over
to the vector elliptic system - linear elasticity. Relying on asymptotic formulas exhibited in our previous work \cite{ChDaFaYo16a}, we analyze two families of shells already investigated in \cite{ArtioliBeiraoHakulaLovadina2009}.
Doing that, we can compare numerical results provided by several different models: The exact Lam\'e model, surfacic models (Love and Naghdi), and our 1D scalar reduction.

The first of these families are cylindrical shells. We show that the lowest eigenvalue\footnote{With the natural frequency $f$, the pulsation $\omega$ and the eigenvalue $\lambda$, we have the relations
\[
   \lambda = \omega^2 = (2\pi f)^2 .
\]}
decays proportionally to the thickness $2\varepsilon$ and that the angular frequency $k$ of its mode tends to infinity like $\varepsilon^{-1/4}$.

The second family is a family of elliptic barrels which we call ``Airy barrels''.  Elliptic means that the two main curvatures (meridian and azimuthal) of the midsurface $\cS$ are non-zero and of the same sign. Airy barrels are characterized by the following relations:
\begin{itemize}
\item The meridian curvature is smaller (in modulus) than the azimuthal curvature at any point of the midsurface $\cS$,
\item The meridian curvature attains its minimum (in modulus) on the boundary of $\cS$.
\end{itemize}
For general elliptic barrels, the first eigenvalue tends to a positive limit $\am_0$ as the thickness tends to $0$. This quantity is proportional to the minimum of the squared meridian curvature. More specifically, for Airy barrels, the azimuthal frequency $k$ is asymptotically proportional to $\varepsilon^{-3/7}$ as $\varepsilon\to0$. Besides, for the particular family of Airy barrels that we study here, a very interesting (and somewhat non-intuitive) phenomenon occurs: For moderately thick barrels $k$ stay equal to $0$ and there is a threshold for $\varepsilon$ below which $k$ has a jump and start growing to infinity.

We start by presenting the angular Fourier transformation in a coordinate independent setting in
sect.\,\ref{s:2} followed by the introduction of the \Bk domains and problems of interest in sect.\,\ref{s:3}. Sect.\,\ref{s:4} is devoted to cases when the angular frequency $k$ of the first mode is zero or converges to a finite value as $\varepsilon\to0$. Cylindrical shells are investigated in sect.\,\ref{s:5} and barrels in sect.\,\ref{s:6}. We conclude in sect.\,\ref{s:7}.

\section{Axisymmetric problems}
\label{s:2}
Before particularizing every object with the help of cylindrical coordinates, we present axisymmetric problems in an abstract setting that exhibits the intrinsic nature of these objects, in particular the angular Fourier coefficients. This intrinsic definition allows to prove that the Fourier coefficients of eigenvectors are eigenvectors if the operator under examination has certain commutation properties.

\subsection{An abstract setting}
Let us consider an axisymmetric domain in $\R^3$. This means that $\Omega$ is invariant by all rotations around a chosen axis $\cA$: For all $\theta\in\T=\R/2\pi\Z$, let $\cR_\theta$ be the rotation of angle $\theta$ around $\cA$. So we assume
\[
   \forall\theta\in\T,\quad \cR_\theta\Omega = \Omega.
\]
Let $\bt=(t_1,t_2,t_3)$ be Cartesian coordinates in $\R^3$. The Laplace operator $\Delta=\partial^2_{t_1}+\partial^2_{t_2}+\partial^2_{t_3}$ is invariant by rotation. This means that for any function $u$
\[
   \forall\theta\in\T,\quad\forall\bt\in\R^3,\quad
   \Delta \big(\ru (\cR_\theta \bt)\big) = (\Delta \ru) (\cR_\theta \bt).
\]
The Lam\'e system $\bL$ of homogeneous isotropic elasticity acts on 3D displacements $\bu=\bu(\bt)$ that are functions with values in $\R^3$. The definition of rotation invariance involves non only rotation of coordinates, but also rotations of displacement vectors. Let us introduce the transformation $\rG_\theta:\bu\mapsto\bv$ between the two displacement vectors $\bu$ and $\bv$
\[
   \rG_\theta(\bu)=\bv \quad\Longleftrightarrow\quad
   \forall\bt,\quad \bv(\bt) = \cR_{-\theta} \big(\bu(\cR_\theta\bt)\big).
\]
Then the Lam\'e system $\bL$ commutes with $\rG_\theta$: For any displacement $\bu$
\begin{equation}
\label{eq:rot1}
   \forall\theta\in\T,\quad \bL(\rG_\theta\bu) = \rG_\theta (\bL\bu).
\end{equation}
Now, in the scalar case, if we define the transformation $\rG_\theta$ by $(\rG_\theta \ru)(\bt)=\ru(\cR_\theta\bt)$, we also have the commutation relation for the Laplacian
\begin{equation}
\label{eq:rot2}
   \forall\theta\in\T,\quad \Delta(\rG_\theta\ru) = \rG_\theta (\Delta\ru).
\end{equation}
The set of transformations $\big(\rG_\theta\big)_{\theta\in\T}$ has a group structure, isomorphic to that of the torus $\T$:
\[
   \rG_\theta \circ \rG_{\theta'} = \rG_{\theta+\theta'},\quad \theta,\theta'\in\T.
\]
The rotation invariance relations \eqref{eq:rot1} and \eqref{eq:rot2} motivate the following angular Fourier transformation $\T\ni\theta\mapsto k\in\Z$ (here $u$ is a scalar or vector function $\ru$ or $\bu$)
\begin{equation}
\label{eq:four}
   \widehat u\ee^k(\bt) = \frac{1}{2\pi} \int_\T (\rG_\varphi u)(\bt)\,e^{-ik\varphi}\,\rd\varphi,
   \quad \bt\in\Omega,\quad k\in\Z.
\end{equation}
Then each Fourier coefficient $\widehat u\ee^k$ enjoys the property:
\begin{equation}
\label{eq:inva}
   (\rG_\theta \widehat u\ee^k)(\bt) = e^{ik\theta}\,\widehat u\ee^k(\bt),\quad \bt\in\Omega,\quad \theta\in\T\,,
\end{equation}
and each pair of Fourier coefficients $\widehat u\ee^k$ and $\widehat u\ee^{k'}$ with $k\neq k'$ satisfies
\[
   \int_\T \widehat u\ee^k(\cR_\theta\bt)\cdot\widehat u\ee^{k'}(\cR_\theta\bt)\,\rd\theta = 0,\quad \bt\in\Omega,
\]
which implies that $\widehat u\ee^k$ and $\widehat u\ee^{k'}$ are  orthogonal along each orbit contained in $\Omega$ and hence in $L^2(\Omega)$. 

The inverse Fourier transform is given by
\begin{equation}
\label{eq:inv}
   u(\bt) =
   \sum_{k\in\Z} \widehat u\ee^k(\bt) ,\quad \bt\in\Omega.
\end{equation}
The function $u$ is said {\em unimodal} if there exists $k_0\in\Z$ such that for all $k\neq k_0$, $\widehat u\ee^k=0$, and $\widehat u\ee^{k_0}\neq0$. Such a function satisfies
\[
   (\rG_\theta u)(\bt) = e^{ik_0\theta}\, u(\bt),\quad \bt\in\Omega,\quad \theta\in\T.
\]

For an operator $\rA$ that commutes with the transformations $\rG_\theta$, i.e., $\rG_\theta\rA=\rA\rG_\theta$, as in \eqref{eq:rot1} and \eqref{eq:rot2}, there holds for any $k\in\Z$
\[
   \widehat{\rA u}{}^k = \frac{1}{2\pi} \int_\T (\rG_\varphi \rA u)\,e^{-ik\varphi}\,\rd\varphi
   = \frac{1}{2\pi} \int_\T (\rA \rG_\varphi u)\,e^{-ik\varphi}\,\rd\varphi = \rA\widehat u\ee^k.
\]
In particular, if $\rA u = \lambda u$, then $\widehat{\rA u}{}^k = \lambda \widehat u\ee^k$. We deduce from the latter equality that
\[
   \lambda \widehat u\ee^k = \rA\widehat u\ee^k.
\]
Therefore any nonzero Fourier coefficient of an eigenvector is itself an eigenvector. We have proved

\begin{lemma}
\label{lem:Fev}
Let $\lambda$ be an eigenvalue of an operator $\rA$ that commutes with the group of transformations $\{\rG_\theta\}_{\theta\in\T}$. Then the associate eigenspace has a basis of unimodal vectors.
\end{lemma}

\begin{remark}
\label{rem:Fev}
If moreover the operator $\rA$ is self-adjoint with real coefficients, the eigenspaces are real. Since for any real function and $k\neq0$, the Fourier coefficient $\widehat u\ee^{-k}$ is the conjugate of $\widehat u\ee^k$, the previous lemma yields that if there is an eigenvector of angular eigenfrequency $k$, there is another one of angular eigenfrequency $-k$ associated with the same eigenvalue.
\end{remark}

\subsection{Cylindrical coordinates}
Let us choose cylindrical coordinates $(r,\varphi,\tau)\in\R_+\times\T\times\R$ associated with the axis $\cA$. This means that $r$ is the distance to $\cA$, $\tau$ an abscissa along $\cA$, and $\varphi$ a rotation angle around $\cA$. We write the change of variables as
\[
   \bt=\sT(r,\varphi,\tau)\quad\mbox{with}\quad
   t_1 = r\cos\varphi,\ \
   t_2 = r\sin\varphi,\ \
   t_3 = \tau\,.
\]
The cylindrical coordinates of the rotated point $\cR_\theta\bt$ are $(r,\varphi+\theta,\tau)$.

\subsubsection{Scalar case}
The Laplace operator in cylindrical coordinates writes
\[
   \Delta = \partial^2_r + \frac{1}{r} \partial_r + \frac{1}{r^2} \partial^2_\varphi + \partial^2_\tau\,.
\]
The classical angular Fourier transform for scalar functions is now
\begin{equation}
\label{eq:fou2}
   \ru^k(r,\tau) = \frac{1}{2\pi} \int_\T \ru(\sT(r,\varphi,\tau)\big)\,e^{-ik\varphi}\,\rd\varphi,
   \quad (r,\tau)\in\omega,\quad k\in\Z,
\end{equation}
where $\omega$ is the meridian domain of $\Omega$. We have the relations
\begin{equation}
\label{eq:rel1}
   \ru^k(r,\tau)\,e^{ik\varphi} =  \widehat \ru\ee^k \big(\sT(r,\varphi,\tau)\big),\quad
   (r,\tau)\in\omega,\quad \varphi\in\T,\quad k\in\Z
\end{equation}
and the classical inverse Fourier formula (compare with \eqref{eq:inv})
\[
   \ru(\sT(r,\varphi,\tau)\big) = \sum_{k\in\Z} \ru^k(r,\tau)\,e^{ik\varphi}\,.
\]
The Laplace operator at frequency $k$ is
\[
   \Delta^{(k)} =
   \partial^2_r + \frac{1}{r} \partial_r - \frac{k^2}{r^2} + \partial^2_\tau\,,
\]
and we have the following diagonalization of $\Delta$
\[
   (\Delta \ru)^k = \Delta^{(k)} \ru^k, \quad k\in\Z.
\]

\subsubsection{Vector case}
Now, an option to find a coordinate basis for the representation of displacements is to consider the partial derivatives of the change of variables $\sT$
\[
   \be_r = \partial_r\sT,\quad
   \be_\varphi = \partial_\varphi\sT,\quad\mbox{and}\quad
   \be_\tau = \partial_\tau\sT.
\]
If we denote by $\be_{t_1}$, $\be_{t_2}$, and $\be_{t_3}$ the orthonormal basis associated with Cartesian coordinates $\bt$, we have
\[
   \be_r = \be_{t_1}\cos\varphi + \be_{t_2}\sin\varphi,\quad
   \be_\varphi = -r\be_{t_1}\sin\varphi + r\be_{t_2}\cos\varphi,
   \quad\mbox{and}\quad \be_\tau=\be_{t_3}\,.
\]
We note the effect of the rotations $\cR_\theta$ on these vectors (we omit the axial coordinate $\tau$)
\begin{equation}
\label{eq:rot3}
   \be_r(r,\varphi+\theta) = (\cR_\theta\be_r)(r,\varphi)
   \quad\mbox{and}\quad
   \be_\varphi(r,\varphi+\theta) = (\cR_\theta\be_\varphi)(r,\varphi)
\end{equation}
and $\be_\tau$ is constant and invariant.

The contravariant components of a displacement $\bu$ in the Cartesian and cylindrical bases are defined such that
\[
   \bu = \ru^{t_1}\be_{t_1} + \ru^{t_2}\be_{t_2} + \ru^{t_3}\be_{t_3}
   = \ru^{r}\be_r+\ru^{\varphi}\be_\varphi + \ru^{\tau}\be_\tau\,.
\]
Covariant components $\ru_j$ are the components of $\bu$ in dual bases. Here we have
\[
   \ru_{t_i} = \ru^{t_i},\ i=1,2,3,\quad\mbox{and}\quad
   \ru_{r} = \ru^{r},\ \ru_{\varphi} = r^2\ru^{\varphi},\ \ru_{\tau} = \ru^{\tau}.
\]
Using relations \eqref{eq:rot3}, we find the representation of transformations $\rG_\theta$
\[
   \rG_\theta\bu (\bt) =
   \ru^{r}(r,\varphi+\theta,\tau)\,\be_r(r,\varphi)
   +\ru^{\varphi}(r,\varphi+\theta,\tau)\,\be_\varphi(r,\varphi)
   + \ru^{\tau}(r,\varphi+\theta,\tau)\,\be_\tau\,,
\]
with $\bt = \sT(r,\varphi,\tau)$. Then the classical Fourier coefficient of a displacement $\bu$ is:
\[
   \bu^k(r,\varphi,\tau) = (\ru^{r})^k(r,\tau)\,\be_r(r,\varphi)
   +(\ru^{\varphi})^k(r,\tau)\,\be_\varphi(r,\varphi)
   + (\ru^{\tau})^k(r,\tau)\,\be_\tau,
\]
where $(\ru^{a})^k$ is the Fourier coefficient given by the classical formula \eqref{eq:fou2} for $\ru=\ru^a$ with $a=r,\varphi,\tau$. We have a relation similar to \eqref{eq:rel1}, valid for displacements:
\begin{equation}
\label{eq:rel2}
   \bu^k(r,\tau)\,e^{ik\varphi} =  \widehat \bu\ee^k \big(\sT(r,\varphi,\tau)\big),\quad
   (r,\tau)\in\omega,\quad \varphi\in\T,\quad k\in\Z.
\end{equation}
Let $\bL$ be the Lam\'e system. When written in cylindrical coordinates in the basis $(\be_r,\be_\varphi,\be_\tau)$, $\bL$ has its coefficients independent of the angle $\varphi$. Replacing the derivative with respect to $\varphi$ by $ik$ we obtain the parameter dependent system $\bL^{(k)}$ that determines the diagonalization of $\bL$
\begin{equation}
\label{eq:Lk}
   (\bL \bu)^k = \bL^{(k)} \bu^k, \quad k\in\Z.
\end{equation}

\section{Axisymmetric shells}

\label{s:3}
We are interested in 3D axisymmetric domains $\Omega=\Omega^\varepsilon$ that are thin in one direction, indexed by their thickness parameter $\varepsilon$.
Such $\Omega^\varepsilon$ is defined by its midsurface $\cS$: We assume that the surface $\cS$ is a smooth bounded connected manifold with boundary in $\R^3$ and that it is orientable, so that there exists a smooth unit normal field $\bp\mapsto\bn(\bp)$ on $\cS$.  For $\varepsilon>0$ small enough the following map is one to one and smooth
\begin{equation}
\label{1E1}
\begin{array}{cccc}
   \Phi : & \cS\times(-\varepsilon,\varepsilon) &\to &\Omega^\varepsilon \\
   & (\bp,x_3) &\mapsto& \bt=\bp+x_3\,\bn(\bp).
\end{array}
\end{equation}
The actual thickness $h$ of $\Omega$ is $2\varepsilon$ (we keep this thickness $h=2\varepsilon$ in mind to bridging with some results of the literature). Such bodies represent (thin) shells in elasticity, whereas they can be called layer domains or thin domains in other contexts.

The boundary of $\Omega^\varepsilon$ has two parts:
\begin{enumerate}
\item Its lateral boundary $\partial_0\Omega^\varepsilon := \Phi\big(\partial\cS\times(-\varepsilon,\varepsilon)\big)$,
\item The rest of its boundary (natural boundary) $\partial_1\Omega^\varepsilon := \partial\Omega^\varepsilon\setminus \partial_0\Omega^\varepsilon$.
\end{enumerate}
The boundary conditions that will be imposed are Dirichlet on $\partial_0\Omega^\varepsilon$ and Neumann on $\partial_1\Omega^\varepsilon$. We consider the two following eigenvalue problems on $\Omega^\varepsilon$, posed in variational form: Let
\[
   V_\Delta(\Omega^\varepsilon) := \{\ru \in H^1(\Omega^\varepsilon) \, , \quad
   \ru = 0 \quad \mbox{on}\quad \partial_0\Omega^\varepsilon\},
\]
and
\[
   V_{\bL}(\Omega^\varepsilon) := \{\bu = (\ru_{t_1},\ru_{t_2},\ru_{t_3}) \in H^1(\Omega^\varepsilon)^3 \, , \quad
   \bu = 0 \quad \mbox{on}\quad \partial_0\Omega^\varepsilon\}.
\]

(i) For the Laplace operator: Find $\lambda\in\R$ and $\ru\in V_\Delta(\Omega^\varepsilon)$, $\ru\neq0$ such that
\begin{equation}
\label{eq:delt}
   \forall\ru^*\in V_\Delta(\Omega_\varepsilon),\quad
   \int_{\Omega^\varepsilon} \nabla \ru \cdot \nabla \ru^* \,\rd\bt =
   \lambda \int_{\Omega^\varepsilon}  \ru \,  \ru^* \,\rd\bt.
\end{equation}
(ii) For the Lam\'e operator: Find $\lambda\in\R$ and $\bu\in V_{\bL}(\Omega^\varepsilon)$, $\bu\neq0$ such that
\begin{equation}
\label{eq:lame}
   \int_{\Omega^\varepsilon} A^{ijlm} e_{ij}(\bu) \,e_{lm}(\tbu) \, \rd \Omega^\varepsilon=
   \lambda \int_{\Omega^\varepsilon} \ru^{t_i} \ru^*_{t_i} \, \rd \Omega^\varepsilon.
\end{equation}
Here we have used the convention of repeated indices, $A_{ijlm}$ is the material tensor associated with the Young modulus $E$ and the Poisson coefficient $\nu$
\begin{equation}
\label{eq:Amat}
   A^{ijlm} = \frac{E\nu}{(1+\nu)(1 - 2\nu)} \delta^{ij}\delta^{lm}
   + \frac{E}{2(1 + \nu)} (\delta^{il}\delta^{jm} + \delta^{im}\delta^{jl}),
\end{equation}
and the covariant components of the train tensor are given by
\[
   e_{ij}(\bu) = \frac12 (\partial_{t_i} \ru_{t_j} + \partial_{t_j} \ru_{t_i}).
\]
The associated $3\times3$ system writes
\[
   \bL = -\frac{E}{2(1+\nu)(1 - 2\nu)} \Big( (1-2\nu)\Delta + \nabla\div \Big).
\]

Both problems \eqref{eq:delt} and \eqref{eq:lame} have discrete spectra and their first eigenvalues are positive. We denote by $\lambda_\Delta(\varepsilon)$ and $\lambda_\La(\varepsilon)$ the smallest eigenvalues of \eqref{eq:delt} and \eqref{eq:lame}, respectively. By Lemma \ref{lem:Fev}, the associate eigenspaces have a basis of unimodal vectors. By Remark \ref{rem:Fev}, in each eigenspace some eigenvectors have a nonnegative angular frequency $k$. We denote by $k_\Delta(\varepsilon)$ and $k_\La(\varepsilon)$ the smallest nonnegative angular frequencies of eigenvectors associated with the the first eigenvalues $\lambda_\Delta(\varepsilon)$ and $\lambda_\La(\varepsilon)$, respectively.

In the next sections, we exhibit cases where the angular frequencies $k(\varepsilon)$ converge to a finite limit as $\varepsilon\to0$ (the quiet cases) and other cases where $k(\varepsilon)$ tends to infinity as $\varepsilon\to0$ (the sensitive or excited cases).

\section{Quiet cases}

\label{s:4}
We know (or reasonably expect) convergence of $k(\varepsilon)$ for the Laplace operator and for plane shells.

\subsection{Laplace operator}
Let us start with an obvious case. Suppose that the shells are plane, i.e. $\cS$ is an open set in $\R^2$. Then $\Omega^\varepsilon$ is a plate The axisymmetry then implies that $\cS$ is a disc or a ring. Let $(x_1,x_2)$ be the coordinates in $\cS$ and $x_3$ be the normal coordinate. In this system of coordinates
\begin{equation}
\label{eq:plate}
   \Omega^\varepsilon =
   \cS \times (-\varepsilon,\varepsilon)
\end{equation}
and the Laplace operator separates variables. One can write
\begin{equation}
\label{eq:sepv}
   \Delta_{\Omega^\varepsilon} =
   \Delta_\cS \otimes \Id_{(-\varepsilon,\varepsilon)} + \Id_\cS \otimes \Delta_{(-\varepsilon,\varepsilon)}\,.
\end{equation}
Here $\Delta_{\Omega^\varepsilon}$ is the 3D Laplacian on $\Omega^\varepsilon$ with Dirichlet conditions on $\partial_0\Omega^\varepsilon$ and Neumann conditions on the rest of the boundary, $\Delta_\cS$ is the 2D Laplacian with Dirichlet conditions on $\partial\cS$, and $\Delta_{(-\varepsilon,\varepsilon)}$ is the 1D Laplacian on $(-\varepsilon,\varepsilon)$ with Neumann conditions in $\pm\varepsilon$. Then the eigenvalues of $\Delta_{\Omega^\varepsilon}$ are all the sums of an eigenvalue of $\Delta_\cS$ and of an eigenvalue of $\Delta_{(-\varepsilon,\varepsilon)}$. The first eigenvalue $\lambda_\Delta(\varepsilon)$ of \eqref{eq:delt} is the first eigenvalue of $-\Delta_{\Omega^\varepsilon}$. Since the first eigenvalue of $-\Delta_{(-\varepsilon,\varepsilon)}$ is $0$, we have
\[
   \lambda_\Delta(\varepsilon) = \lambda_\cS
\]
and the corresponding eigenvector is $\ru(x_1,x_2,x_3)=\rv(x_1,x_2)$ where $(\lambda_\cS,\rv)$ is the first eigenpair of $-\Delta_\cS$. Thus $k_\Delta(\varepsilon)$ is independent of $\varepsilon$, and is the angular frequency of $\rv$.

In the case of a shell that is not a plate, the equality \eqref{eq:sepv} is no more true. However, if $\Delta_\cS$ denotes now the Laplace-Beltrami on the surface $\cS$ with Dirichlet boundary condition, an extension of the result\footnote{In \cite{Schatzman96a}, the manifold $\cS$ (denoted there by $M$) is without boundary. We are convinced that all proofs can be extended to the Dirichlet lateral boundary conditions when $\cS$ has a smooth boundary.} of \cite{Schatzman96a} yields that the smallest eigenvalue of the right-hand side of \eqref{eq:sepv} should converge to the smallest eigenvalue of $\Delta_{\Omega^\varepsilon}$. An extension of \cite[Th.\,4]{Schatzman96a} gives, more precisely, that
\begin{equation}
\label{eq:asyd}
   \lambda_\Delta(\varepsilon) = \lambda_\cS + \am_1 \varepsilon + \cO(\varepsilon^2),
   \quad\mbox{as}\ \ \varepsilon\to0,
\end{equation}
for some coefficient $\am_1$ independent on $\varepsilon$.
Concerning the angular frequency $k_\Delta(\varepsilon)$, a direct argument allows to conclude.

\begin{lemma}
\label{lem:lap}
Let $\Omega^\varepsilon$ be an axisymmetric shell. The first eigenvalue \eqref{eq:delt} of the Laplace operator is simple and $k_\Delta(\varepsilon)=0$.
\end{lemma}

\begin{proof}
The simplicity of the first eigenvalue of the Laplace operator with Dirichlet boundary conditions is a well-known result. Here we reproduce the main steps of the arguments (see e.g.\ \cite[sect.\,7.2]{HelfferBook13}) to check that this result extends to more general boundary conditions.

Let $\ru$ be an eigenvector associated with the first eigenvalue $\lambda$. A consequence of the Kato equality
\[
   \nabla |\ru| = {\rm sgn}\ee(\ru) \nabla\ru \quad\mbox{almost everywhere}\footnote{With the convention that ${\rm sgn}\ee(\ru) = 0$ when $\ru = 0$.}
\]
is that $|\ru|$ satisfies the same boundary conditions as $\ru$ and the same eigenequation $-\Delta|\ru|=\lambda|\ru|$ as $\ru$. Therefore $|\ru|$ is an eigenvector with constant sign. The equation $-\Delta|\ru|=\lambda|\ru|$ implies that $-\Delta|\ru|$ is nonnegative, and hence $|\ru|$ satisfies the mean value property
\[
   |\ru(\bt_0)| \ge \frac{1}{\meas(B(\bt_0,\rho))}\int_{B(\bt_0,\rho)} |\ru(\bt)|\,\rd\bt
\]
for all $\bt_0\in\Omega^\varepsilon$ and all $\rho>0$ such that the ball $B(\bt_0,\rho)$ is contained in $\Omega^\varepsilon$. Hence $|\ru|$ is positive everywhere in $\Omega^\varepsilon$. Therefore $\ru=\pm|\ru|$ and we deduce that the first eigenvalue is simple.

By Lemma \ref{lem:Fev}, this eigenvector is unimodal. Let $k$ be its angular frequency. If $k\neq0$, by Remark \ref{rem:Fev} there would exist an independent eigenvector of angular frequency $-k$ for the same eigenvalue. Therefore $k=0$.
\end{proof}

\subsection{Lam\'e system on plates}
The domain $\Omega^\varepsilon$ is the product \eqref{eq:plate} of $\cS$ by $(-\varepsilon,\varepsilon)$. For the smallest eigenvalues $\lambda_\La(\varepsilon)$ of the Lam\'e problem \eqref{eq:lame}, we have the convergence result of \cite[Th.8.1]{DDFR99}
\begin{equation}
\label{eq:asyL}
   \lambda_\La(\varepsilon) = \lambda_{\B} \,\varepsilon^2
   + \cO(\varepsilon^3),
   \quad\mbox{as}\ \ \varepsilon\to0,
\end{equation}
Here, $\lambda_{\B}$ is the first Dirichlet eigenvalue of the scalar bending operator $\B$ that, in the case of plates, is simply a multiple of the bilaplacian (Kirchhoff model)
\begin{equation}
\label{eq:bilap}
   \B = \frac{1}{3}\frac{E}{1-\nu^2}\,\Delta^2\quad\mbox{on}\quad H^2_0(\cS).
\end{equation}
The reference  \cite[Th.8.2]{DDFR99} proves convergence also for eigenvectors. In particular the normal component $\ru^3$ of a suitably normalized eigenvector converges to a Dirichlet eigenvector of $\B$. This implies that the angular frequency $k_\rL(\varepsilon)$ converges to the angular frequency $k_\B$ of the first eigenvector of the bending operator $\B$.

\subsection{Lam\'e system on a spherical cap}
A spherical cap $\Omega^\varepsilon$ can be easily defined in spherical coordinates $(\rho,\theta,\varphi)\in[0,\infty)\times[-\frac{\pi}{2},\frac{\pi}{2}]\times\T$ (radius, meridian angle, azimuthal angle) as
\[
   \Omega^\varepsilon = \big\{\bt\in\R^3,\quad \rho\in(R-\varepsilon,R+\varepsilon),\quad \varphi\in\T,\quad
   \theta\in(\Theta,\frac{\pi}{2}]\big\}.
\]
Here $R>0$ is the radius of the midsurface $\cS$ and $\Theta\in(0,\pi)$ is a given meridian angle. Numerical experiments conducted in \cite[sect.6.4.2]{DaFaYo} exhibited convergence for the first eigenpair as $\varepsilon\to0$ (when $\Theta=\frac{\pi}{4}$), see Fig.10 {\em loc. cit.}. We do not have (yet) any theoretical proof for this.

\section{Sensitive cases: Developable shells}
\label{s:5}
Developable shells have one main curvature equal to $0$. Excluding plates that are considered above, we see that we are left with cylinders and cones\footnote{Since we consider here shells with a {\em smooth} midsurface, cones should be trimmed so that they do not touch the rotation axis.}.

The case of cylinders was addressed in the literature with different levels of precision. In cylindrical coordinates $(r,\varphi,\tau)\in[0,\infty)\times\T\times\R$ (radius, azimuthal angle, axial abscissa) a thin cylindrical shell is defined as
\[
   \Omega^\varepsilon = \big\{\bt\in\R^3,\quad r\in(R-\varepsilon,R+\varepsilon),\quad \varphi\in\T,\quad
   \tau\in(-\tfrac{L}{2},\tfrac{L}{2})\big\}.
\]
Here $R>0$ is the radius of the midsurface $\cS$ and $L$ its length. The lateral boundary of $\Omega^\varepsilon$ is
\[
   \partial_0\Omega^\varepsilon  = \big\{\bt\in\R^3,\quad r\in(R-\varepsilon,R+\varepsilon),\quad \varphi\in\T,\quad
   \tau = \pm\tfrac{L}{2}\big\}.
\]

One may find in \cite{SoedelBook,SoedelEncyclopedia} an example of analytic calculation for a {\em simply supported} cylinder using a simplified shell model (called Donnel-Mushtari-Vlasov). We note that simply supported conditions on the lateral boundary of a cylinder allow reflection across this lateral boundary, so that separation of the three variables using trigonometric Ansatz functions is possible. This example shows that for $R=1$, $L=2$ and $h=0.02$ (i.e., $\varepsilon=0.01$) the smallest eigenfrequency does not correspond to a {\em simple} eigenmode, i.e., a mode for which $k=0$, but to a mode with $k=4$.

In figure \ref{f:M1} we plot numerical dispersion curves of the {\em exact} Lam\'e model $\bL$ for several values of the thickness $h=2\varepsilon$ ($0.1$, $0.01$, and $0.001$). This means that we discretize the exact 2D Lam\'e model $\bL^{(k)}$ obtained after angular Fourier transformation, see \eqref{eq:Lk}, on the meridian domain
\[
   \omega^\varepsilon = \big\{(r,\tau)\in\R_+\times\R,\quad r\in(R-\varepsilon,R+\varepsilon),\quad
   \tau\in(-\tfrac{L}{2},\tfrac{L}{2})\big\}.
\]
We compute by a finite element method for a collection of values of $k\in\{0,1,\ldots,k_{\max}\}$ so that for each $\varepsilon$, we have reached the minimum in $k$ for the first eigenvalue.

\begin{figure}[ht]
\includegraphics[scale=0.6]{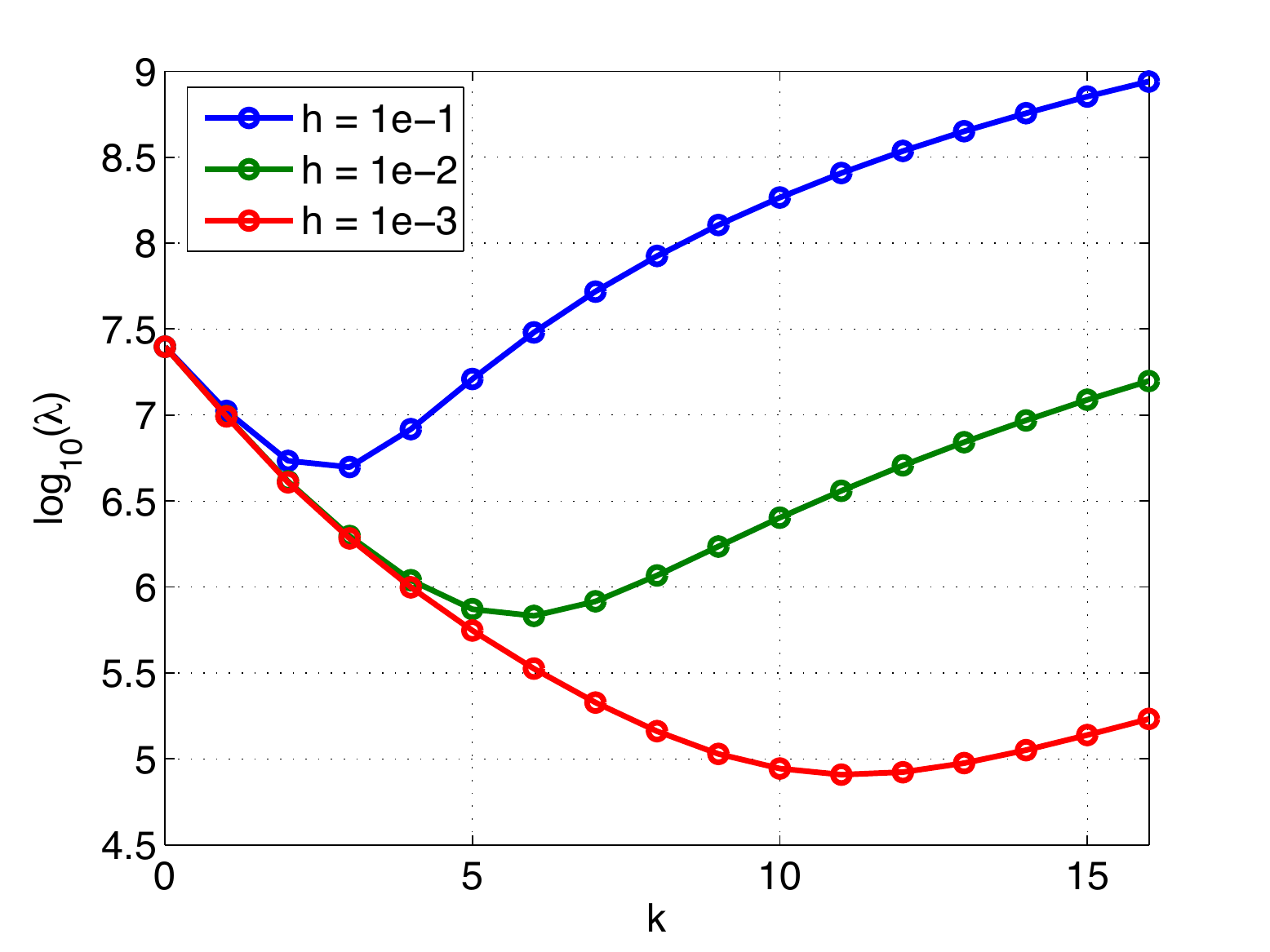}
   \caption{\small Cylinders with $R=1$ and $L=2$: $\log_{10}$ of first eigenvalue of $\bL^{(k)}$ depending on $k$ for several values of the thickness $h=2\varepsilon$. Material constants $E = 2.069 \cdot10^{11}$, $\nu=0.3$, and $\rho=7868$ as in \cite{ArtioliBeiraoHakulaLovadina2009}}
\label{f:M1}\end{figure}

So we see that the minimum is attained for $k=3$, $k=6$, and $k=11$ when $h=0.1$, $0.01$, and $0.001$, respectively. We have also performed direct 3D finite element computations for the same values of the thickness and obtained coherent results. In figures \ref{f:M2}-\ref{f:M4} we represent the shell without deformation and the radial component of the first eigenvector for the three values of the thickness.

\begin{figure}[ht]
\includegraphics[scale=0.25]{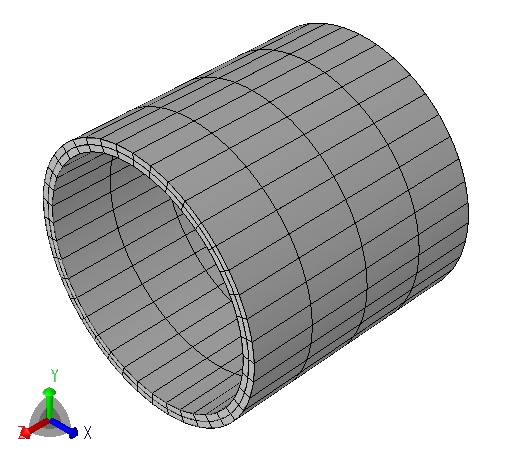}%
\includegraphics[scale=0.25]{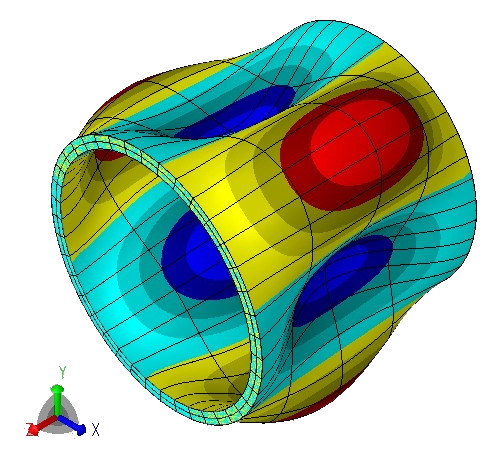}%
   \caption{\small Cylinder with $R=1$, $L=2$ and $h=10^{-1}$: First eigenmode (radial component).}
\label{f:M2}\end{figure}

\begin{figure}[ht]
\includegraphics[scale=0.25]{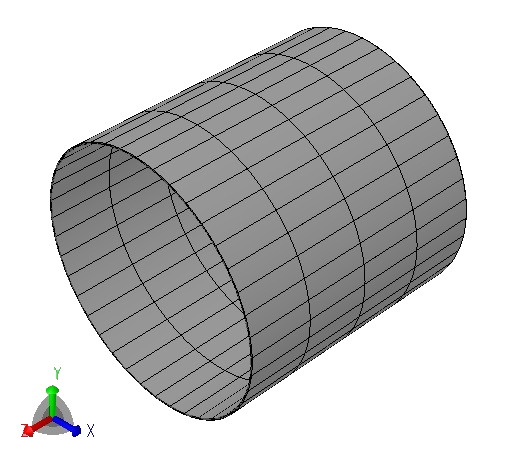}%
\includegraphics[scale=0.25]{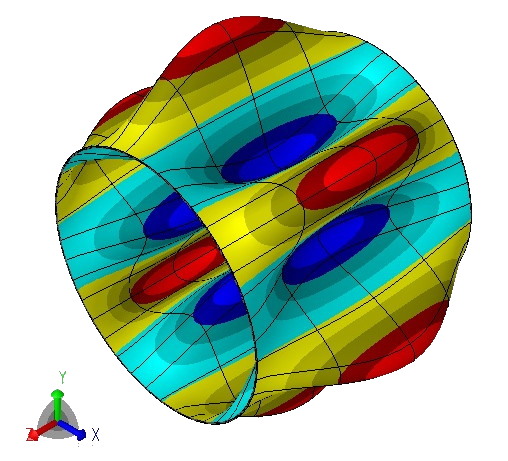}%
   \caption{\small Cylinder with $R=1$, $L=2$ and $h=10^{-2}$: First eigenmode (radial component).}
\label{f:M3}\end{figure}

\begin{figure}[ht]
\includegraphics[scale=0.25]{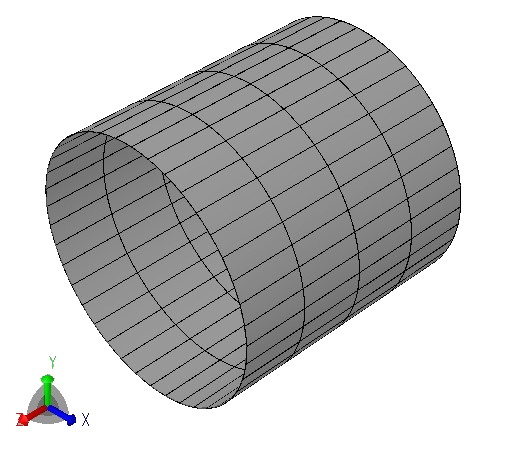}%
\includegraphics[scale=0.25]{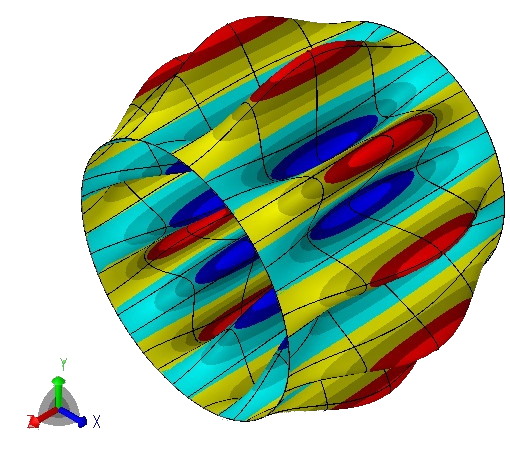}%
   \caption{\small Cylinders with $R=1$, $L=2$ and $h=10^{-3}$: First eigenmode (radial component).}
\label{f:M4}\end{figure}

In fact the first 3D eigenvalue $\lambda_\La(\varepsilon)$ and its associated angular frequency $k_\rL(\varepsilon)$ follow precise power laws that can be determined. A first step in that direction is the series of papers by Artioli, Beir{\~a}o Da Veiga, Hakula and Lovadina \cite{BeiraoLovadina2008,ArtioliBeiraoHakulaLovadina2008,ArtioliBeiraoHakulaLovadina2009}. In these papers the authors investigate the first eigenvalue of classical {\em surface models} posed on the midsurface $\cS$. Such models have the form
\begin{equation}
\label{chopin}
    \bK(\varepsilon) = \bM + \varepsilon^2 \bB.
\end{equation}
The simplest models are $3\times3$ systems. The operator $\bM$ is the membrane operator and $\bB$ the bending operator. These models are obtained using the assumption that normals to the surface in $\Omega^\varepsilon$ are transformed in normals to the deformed surfaces. In the mathematical literature the Koiter model \cite{Koiter5,Koiter2} seems to be the most widely used, while in the mechanical engineering literature so-called Love-type equations will be found \cite{SoedelEncyclopedia}. These models differ from each other by lower order terms in the bending operator $\bB$. As we will specify later on, this difference has no influence in our results.

Defining the order $\alpha$ of a positive function $\varepsilon\mapsto\lambda(\varepsilon)$, continuous on $(0,\varepsilon_0]$, by the conditions
\begin{equation}
\label{order}
   \forall\eta>0,\quad
   \lim_{\varepsilon\to0^+} \lambda(\varepsilon)\,\varepsilon^{-\alpha+\eta} = 0
   \quad\mbox{and}\quad
   \lim_{\varepsilon\to0^+} \lambda(\varepsilon)\,\varepsilon^{-\alpha-\eta} = \infty
\end{equation}
\cite{BeiraoLovadina2008,ArtioliBeiraoHakulaLovadina2008,ArtioliBeiraoHakulaLovadina2009} proved that $\alpha=1$ for the first eigenvalue of $\bK(\varepsilon)$ in clamped cylindrical shells.
They also investigated by numerical simulations the azimuthal frequency $k(\varepsilon)$ of the first eigenvector of $\bK(\varepsilon)$ and identified power laws of type $\varepsilon^{-\beta}$ for $k(\varepsilon)$. They found $\beta=\frac14$ for cylinders (see also \cite{BeiraoHakulaPitkaranta2008} for some theoretical arguments based on special Ansatz functions in the axial direction).

In \cite{ChDaFaYo16a}, we constructed analytic formulas that are able to provide an asymptotic expansion for $k(\varepsilon)$ and $\lambda(\varepsilon)$, and consequently for $k_\La(\varepsilon)$ and $\lambda_\La(\varepsilon)$:
\begin{equation}
\label{eq:asycyl}
   k(\varepsilon) \simeq \gamma\varepsilon^{-1/4}\quad\mbox{and}\quad
   \lambda(\varepsilon) \simeq \am_1\varepsilon\,,
\end{equation}
with explicit expressions of $\gamma$ and $\am_1$ using the material parameters $E$, $\rho$ and $\nu$, the sizes $R$ and $L$ of the cylinder, and the first eigenvalue $\mu^{\sf bilap}\simeq500.564$ of the bilaplacian $\eta\mapsto\partial^4_z\eta$ on the unit interval $(0,1)$ with Dirichlet boundary conditions $\eta(0)=\eta'(0)=\eta(1)=\eta'(1)=0$,
cf \cite[sect.\,5.2.2]{ChDaFaYo16a}:
\begin{equation}
\label{eq:ga1cyl}
   \gamma =  \bigg(\frac{R^6}{L^4}\,3(1-\nu^2)\, \mu^{\sf bilap}\bigg)^{1/8}
   \quad\mbox{and}\quad
   \am_1 = \frac{2E}{\rho RL^2} \sqrt{\frac{ \mu^{\sf bilap}}{3(1-\nu^2)}}\;.
\end{equation}

We compare the asymptotics \eqref{eq:asycyl}-\eqref{eq:ga1cyl} with the computed values of $k_\La(\varepsilon)$ by 2D and 3D FEM discretizations, see figure \ref{f:M5}. The values of $k_\La(\varepsilon)$ are determined for each value of the thickness:
\begin{itemize}
\item In 2D, by the abscissa of the minimum of the dispersion curve (see figure \ref{f:M1})
\item In 3D, by the number of angular oscillations of the first mode (see figures \ref{f:M2}-\ref{f:M4})
\end{itemize}

\begin{figure}[ht]
\includegraphics[scale=0.6]{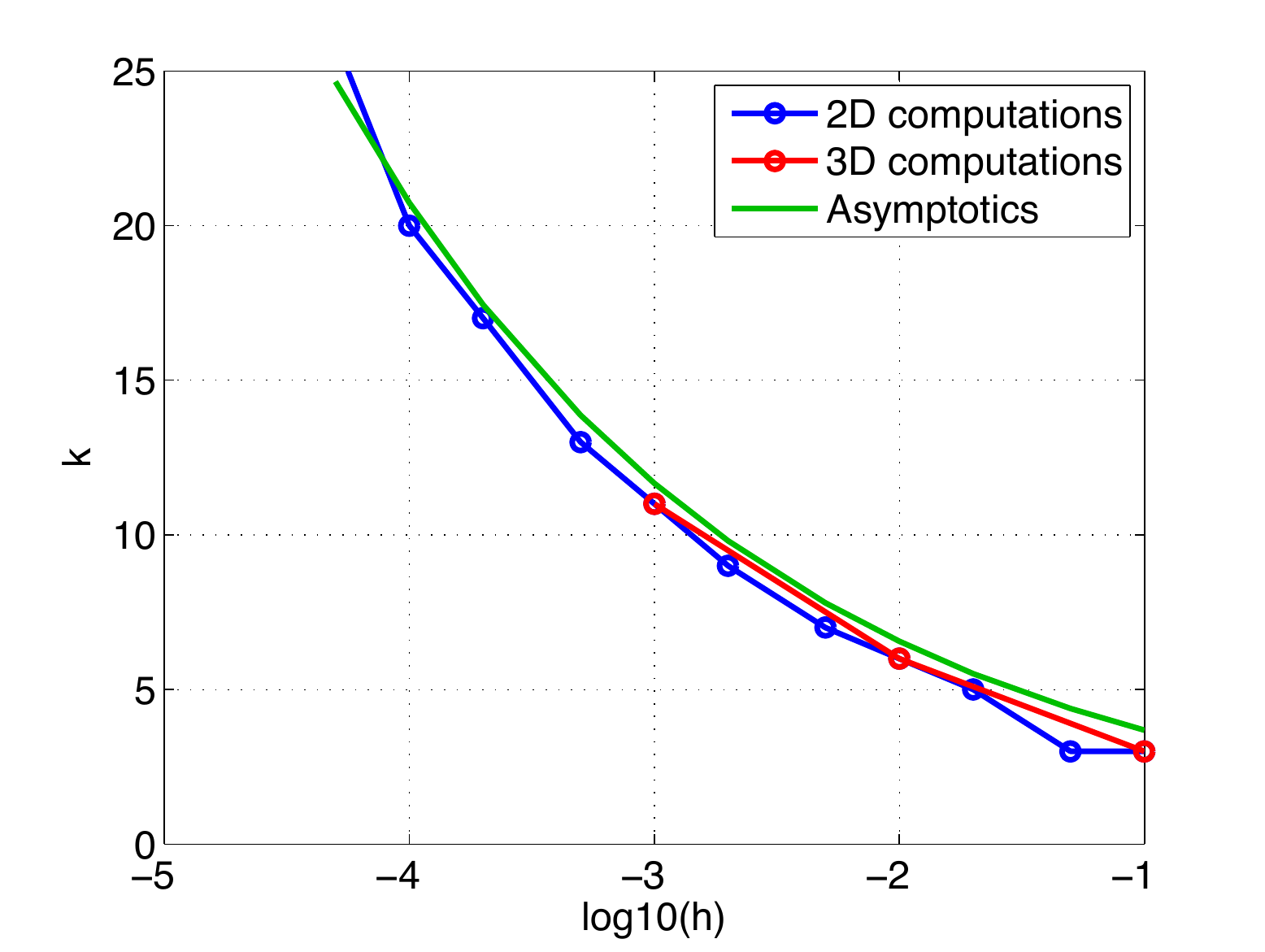}
   \caption{\small Cylinders ($R=1$, $L=2$): Computed values of $k_\La(\varepsilon)$ versus the thickness $h=2\varepsilon$. The asymptotics is $h\mapsto 9.2417\cdot\varepsilon^{-1/4}\simeq 11\cdot h^{-1/4}$ (with $\nu=0.3$).}
\label{f:M5}\end{figure}

Finally we compare the asymptotics \eqref{eq:asycyl}-\eqref{eq:ga1cyl} with the computed eigenvalues $\lambda_\rL(\varepsilon)$ by four different methods, see figure \ref{f:M6}.

\begin{figure}[ht]
\includegraphics[scale=0.6]{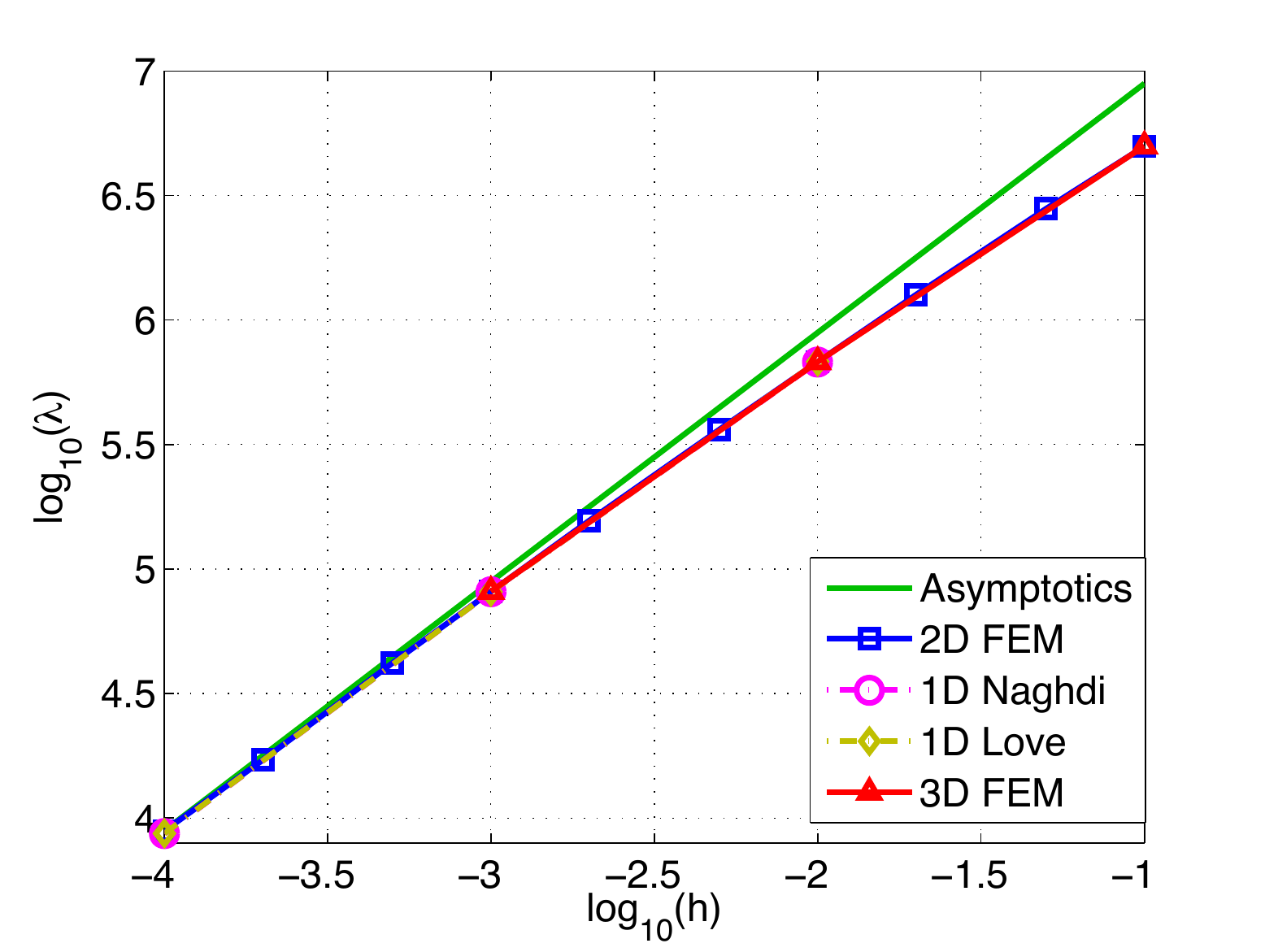}
   \caption{\small Cylinders ($R=1$, $L=2$): Computed values of $\lambda_\rL(\varepsilon)$ versus the thickness $h=2\varepsilon$. Material constants $E = 2.069 \cdot10^{11}$, $\nu=0.3$, and $\rho=7868$. 1D Naghdi and Love models are computed in \cite{ArtioliBeiraoHakulaLovadina2009}. The asymptotics is $h\mapsto 6.770\cdot\varepsilon\,E/\rho= 3.385\cdot h\,E/\rho$.}
\label{f:M6}\end{figure}

Problems considered in figure \ref{f:M6}:
\begin{itemize}
\item[a)] Lam\'e system $\bL^{(k)}$ on the meridian domain $\omega^\varepsilon$, computed for a collection of values of $k$ by 2D finite element method.
\item[b)] Naghdi model \cite{Naghdi1} on the meridian set $\cC=(-1,1)$ of the midsurface,  computed for a collection of values of $k$ by 1D finite element method in \cite{ArtioliBeiraoHakulaLovadina2009}.
\item[c)] Love-type model \cite{SoedelEncyclopedia} on $\cC$, computed for a collection of values of $k$ by collocation in \cite{ArtioliBeiraoHakulaLovadina2009}.
\item[d)] 3D finite element method on the full domain $\Omega^\varepsilon$.
\end{itemize}

In methods a), b) and c),
\[
   \lambda(\varepsilon) = \min_{0\le k\le k_{\max}} \lambda^{(k)}(\varepsilon)
\]
where $\lambda^{(k)}(\varepsilon)$ is the first eigenvalue of the problem with angular Fourier parameter $k$ (remind that $ \lambda^{(k)}(\varepsilon)= \lambda^{(-k)}(\varepsilon)$). In method d), $\lambda(\varepsilon)$ is the first eigenvalue.

We can observe that these four methods yield very similar results and that the agreement with the asymptotics is quite good.
In \cite[sect.\,5]{ChDaFaYo16a} the case of trimmed cones is handled in a similar way and yields goods results, too.

\section{A sensitive family of elliptic shells, the Airy barrels}
\label{s:6}

In this section we consider a family of shells defined by a parametrization with respect to the axial coordinate, which is denoted by $z$ when it plays the role of a parametric variable: The meridian curve $\cC$ of the surface $\cS$ is defined in the half-plane $\R_+\times\R$ by
\[
   \cC = \big\{(r,z)\in\R_+\times\R,\quad z\in\cI,\ r=f(z)\big\}
\]
where $\cI$ is a chosen bounded interval and $f$ is a smooth function on the closure of $\cI$. We assume that $f$ is positive on $\overline\cI$. Then the midsurface is parametrized as (with values in Cartesian variables)
\begin{equation}
\label{eq:surfpara}
\begin{array}{ccccc}
    & \cI\times\T &\longrightarrow &\cS \\
    & (z,\varphi)  &\longmapsto& (t_1,t_2,t_3) = (f(z)\cos\varphi, \,f(z)\sin\varphi,\,z).
\end{array}
\end{equation}
Finally, the transformation $\F:(z,\varphi,x_3)\mapsto(t_1,t_2,t_3)$ sends the product $\cI\times\T\times(-\varepsilon,\varepsilon)$ onto the shell $\Omega^\varepsilon$ and is explicitly given by
\begin{equation}
\label{eq:para}
   t_1 = \big(f(z) + x_3\ \tfrac{1}{s(z)} \big) \cos\varphi,\quad
   t_2 = \big(f(z) + x_3\ \tfrac{1}{s(z)} \big) \sin\varphi,\quad
   t_3 = z - x_3 \ \tfrac{f'(z)}{s(z)},
\end{equation}
where $s$ is the curvilinear abscissa
\[
   s(z) = \sqrt{1+f'^2(z)} .
\]
With shells parametrized in such a way, we are in the elliptic case (that means a positive Gaussian curvature) if and only if $f''$ is negative on $\overline\cI$. In this same situation the references \cite{BeiraoLovadina2008,ArtioliBeiraoHakulaLovadina2008,ArtioliBeiraoHakulaLovadina2009} proved that the order \eqref{order} of the first eigenvalue of $\bK(\varepsilon)$ is $\alpha=0$. In \cite{ArtioliBeiraoHakulaLovadina2009}, numerical simulations are presented for the case
\begin{equation}
\label{eq:ell}
   f(z) = 1 - \frac{z^2}{2} \quad\mbox{on}\quad
   \cI = (-0.892668,0.892668),
\end{equation}
by solving the Naghdi and the Love models. A power law $k(\varepsilon)\sim \varepsilon^{-2/5}$ is suggested for the angular frequency of the first mode. The shells defined by \eqref{eq:para}-\eqref{eq:ell} have the shape of barrels, figure \ref{f:N2}.

\begin{figure}[ht]
\includegraphics[scale=0.2]{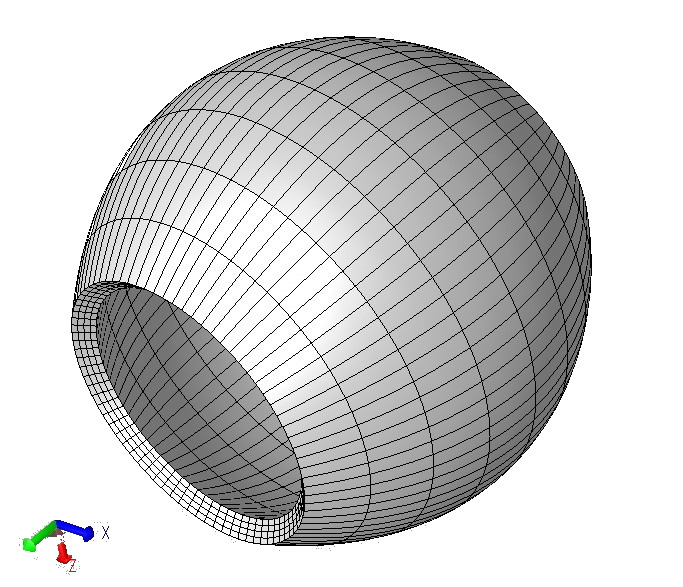}%
   \caption{\small Shell \eqref{eq:para}-\eqref{eq:ell} with $h=10^{-1}$.}
\label{f:N2}\end{figure}

Before presenting the analytical formulas of the asymptotics \cite{ChDaFaYo16a}, let us show results of our 2D and 3D FEM computations. In figure \ref{f:N1} we plot numerical dispersion curves of the exact Lam\'e model $\bL$ for several values of the thickness $h=2\varepsilon$ ($0.01$, $0.004$, $0.002$ and $0.001$).

\begin{figure}[ht]
\includegraphics[scale=0.6]{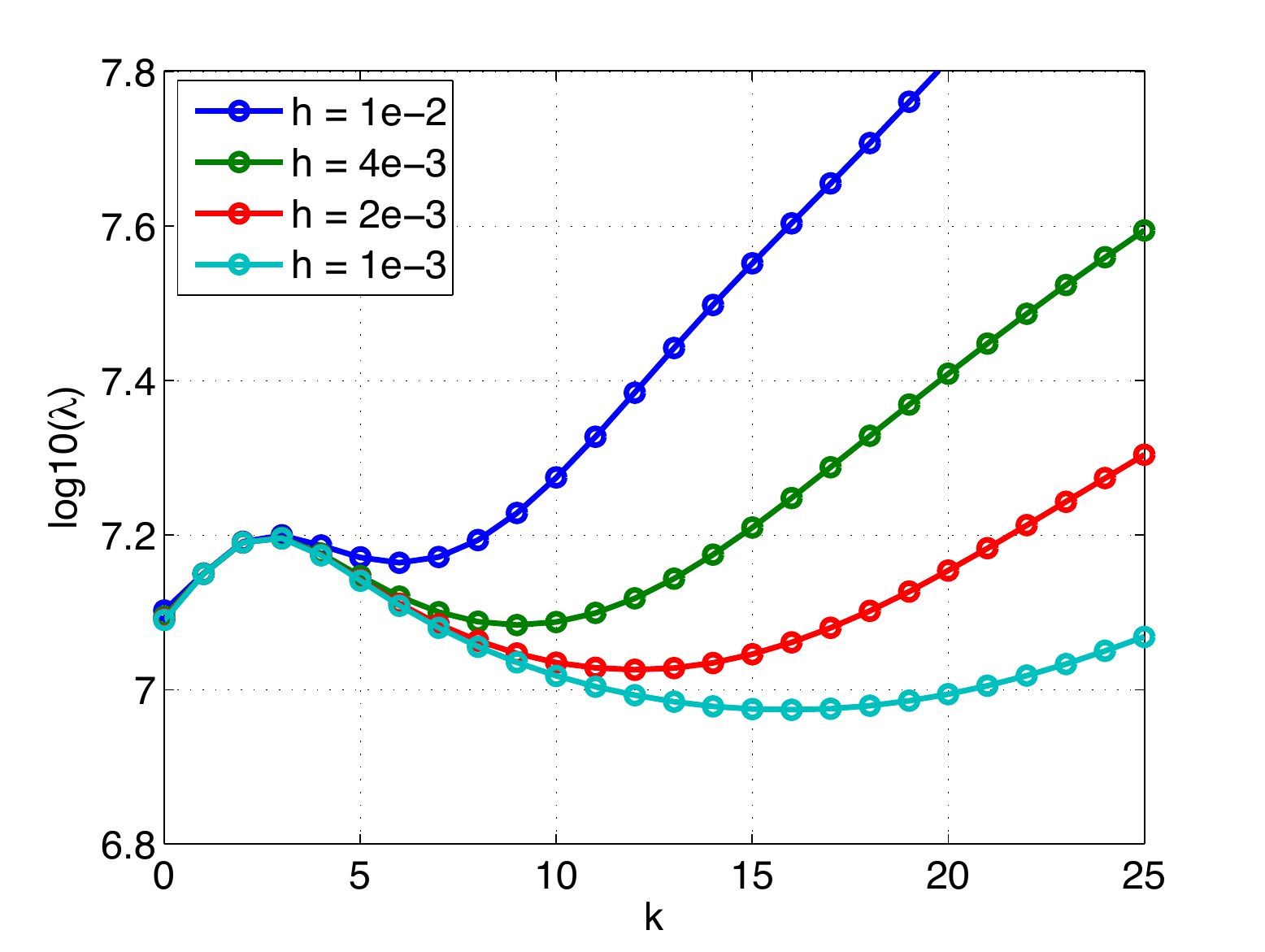}
   \caption{\small Shell \eqref{eq:para}-\eqref{eq:ell}: $\log_{10}$ of first eigenvalue of $\bL^{(k)}$ depending on $k$ for several values of the thickness $h=2\varepsilon$. Material constants as in \cite{ArtioliBeiraoHakulaLovadina2009}}
\label{f:N1}\end{figure}
We can see that, in contrast with the cylinders, $k=0$ is a local minimum of all dispersion curves. This minimum is global when $h\ge0.005$. A second local minimum shows up, which becomes the global minimum when $h\le0.004$ ($k=9$, $12$ and $16$ for $h=0.004$, $0.002$, and $0.001$, respectively. In figure \ref{f:N3} we represent the radial component of the first eigenvector for these four values of the thickness obtained by direct 3D FEM.

\begin{figure}[ht]
\includegraphics[scale=0.2]{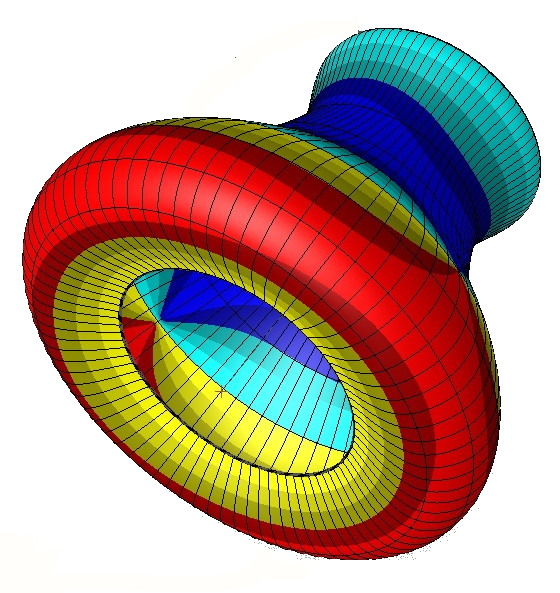}
\includegraphics[scale=0.21]{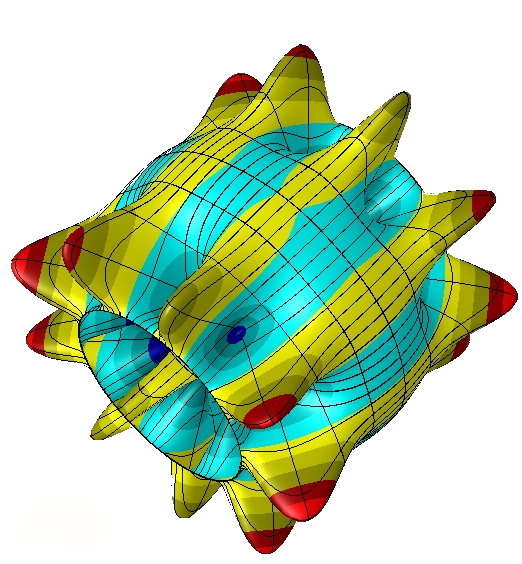}
\includegraphics[scale=0.168]{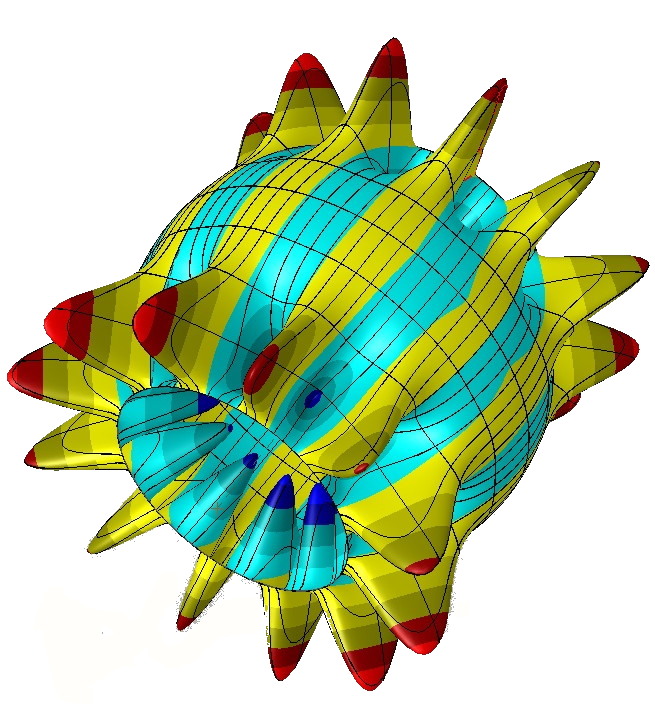}
\includegraphics[scale=0.224]{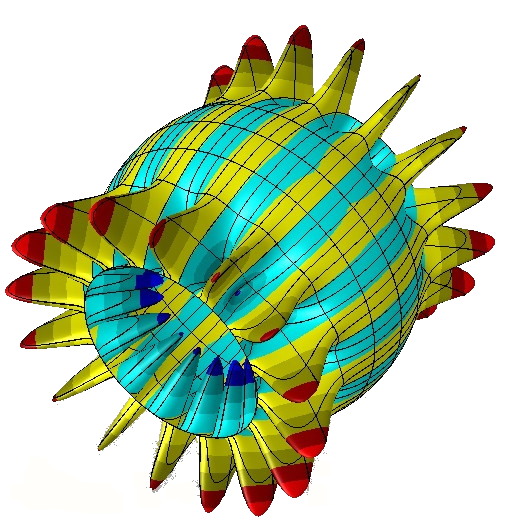}
   \caption{\small Shell \eqref{eq:para}-\eqref{eq:ell} and $h=0.01,\;0.004,\;0.002,\;0.001$: First eigenmode (radial component).}
\label{f:N3}\end{figure}

Comparing with the cylindrical case, we observe a new phenomenon: the eigenmodes also concentrate in the meridian direction, close to the ends of the barrel, displaying a boundary layer structure as $\varepsilon\to0$. In \cite{ChDaFaYo16a} we have classified elliptic shells according to behavior of the function (proportional to the square of the meridian curvature $b^z_z$)
\begin{equation}
\label{eq:H0}
   \rH_0 = \frac{E}{\rho}\, \frac{f''^2}{(1+f'^2)^3}\,.
\end{equation}
If $\rH_0$ is not constant, the classification depends on the localization of the minimum of $\rH_0$. If the minimum is attained in a point $z_0$ that is at one end of $\cI$, we are in what we called the {\em Airy case}. We observe that for the function $f=1-\frac{z^2}{2}$
\[
   \rH_0 = \frac{E}{\rho}\, \frac{1}{(1+z^2)^3} \,.
\]
Its minimum is clearly attained at the two ends $\pm z_0$ of the symmetric interval $\cI$. In the Airy case our asymptotic formulas take the form, see \cite[sect.\,6.4]{ChDaFaYo16a},
\begin{equation}
\label{eq:asyell}
   k(\varepsilon) \simeq \gamma\varepsilon^{-3/7}\quad\mbox{and}\quad
   \lambda(\varepsilon) \simeq \am_0 + \am_1\varepsilon^{2/7} \quad\mbox{with}\quad \am_0=\rH_0(\pm z_0)\,.
\end{equation}
To give the values of $\gamma$ and $\am_1$ we need to introduce the functions
\begin{equation}
\label{eq:B0g}
   \rh(z) = -\frac{2E}{\rho}\Big(\frac{ff''}{s^6}+\frac{f^2f''^2}{s^8}\Big)(z)
   \quad\mbox{and}\quad
   \rB_0(z) =\frac{E}{\rho} \,\frac{1}{3(1-\nu^2)} \,\frac{1}{f(z)^4}\,.
\end{equation}
With
\begin{equation}
\label{eq:bc}
   \rb=\rB_0(z_0) \quad\mbox{and}\quad
   \rc=\mathsf{z}_{\mathsf Airy}^{(1)} \, \big( \rh(z_0)\big)^{1/3}\,
   \big(\partial_z\rH_0(z_0)\big)^{2/3},
\end{equation}
(here $\mathsf{z}_{\mathsf Airy}^{(1)}\simeq 2.33810741$ is the first zero of the reverse Airy function) we have
\begin{equation}
\label{eq:ga1ell}
   \gamma = \Big(\frac{\rc}{6\rb}\Big)^{3/14}\quad\mbox{and}\quad
   \am_1 = (6\rb \rc^{6})^{1/7} (1 + \frac16) \,.
\end{equation}
We compare the asymptotics \eqref{eq:asyell}-\eqref{eq:ga1ell} with the computed values of $k_\La(\varepsilon)$ by 2D and 3D FEM discretizations, see figure \ref{f:N5}. The values of $k_\La(\varepsilon)$ are determined for each value of the thickness by the same numerical methods as in the cylindric case.

\begin{figure}[ht]
\includegraphics[scale=0.6]{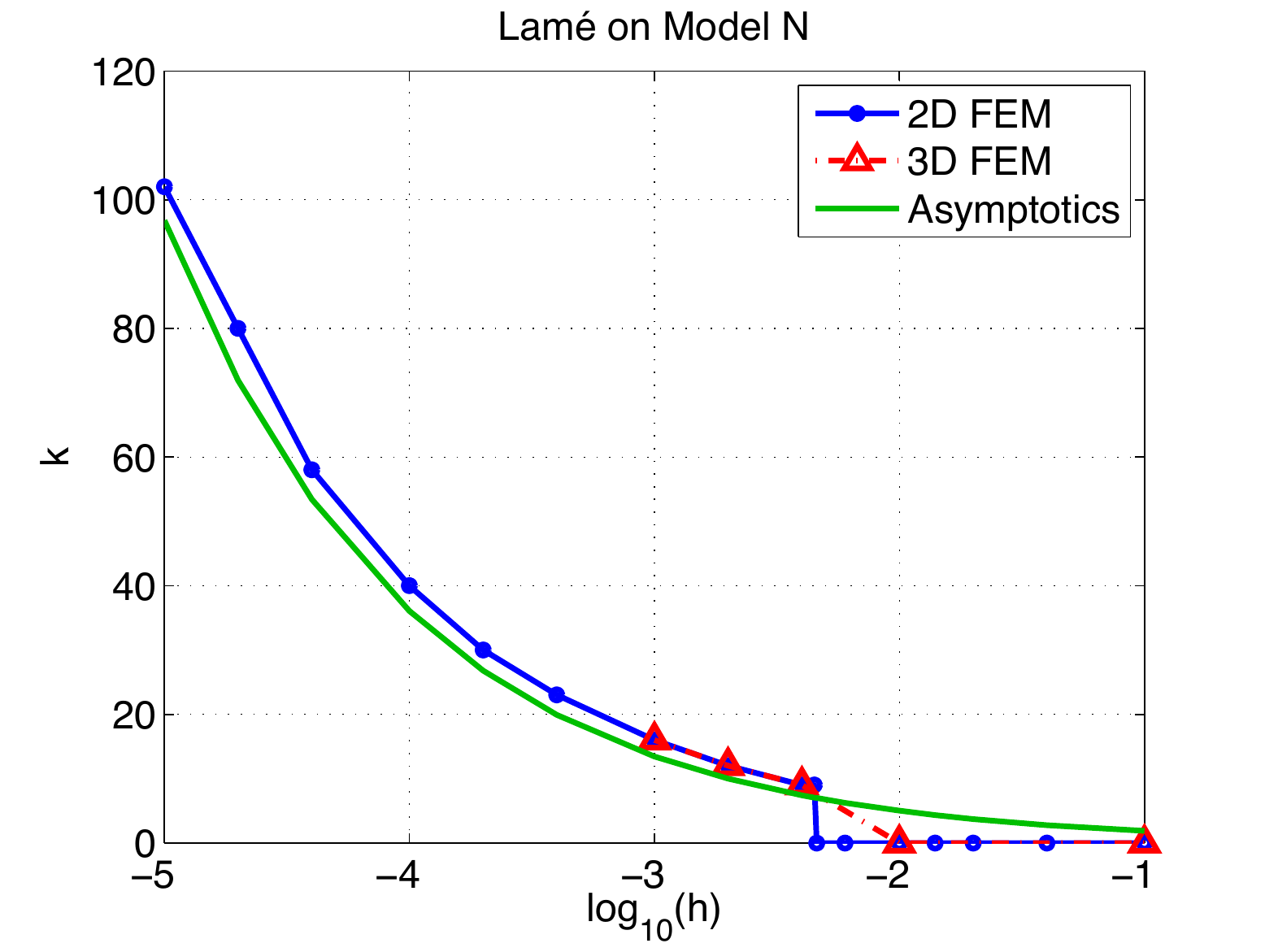}
   \caption{\small Shell \eqref{eq:para}-\eqref{eq:ell}: Computed values of $k_\La(\varepsilon)$ versus the thickness $h=2\varepsilon$. The asymptotics is $h\mapsto 0.51738\cdot\varepsilon^{-3/7}\simeq 0.6963\cdot h^{-3/7}$ (with $\nu=0.3$).}
\label{f:N5}\end{figure}

Finally we compare the asymptotics \eqref{eq:asyell}-\eqref{eq:ga1ell} with the computed eigenvalues $\lambda_\rL(\varepsilon)$ by the same four different methods as in the cylinder case, see figure \ref{f:N6}.

\begin{figure}[ht]
\includegraphics[scale=0.6]{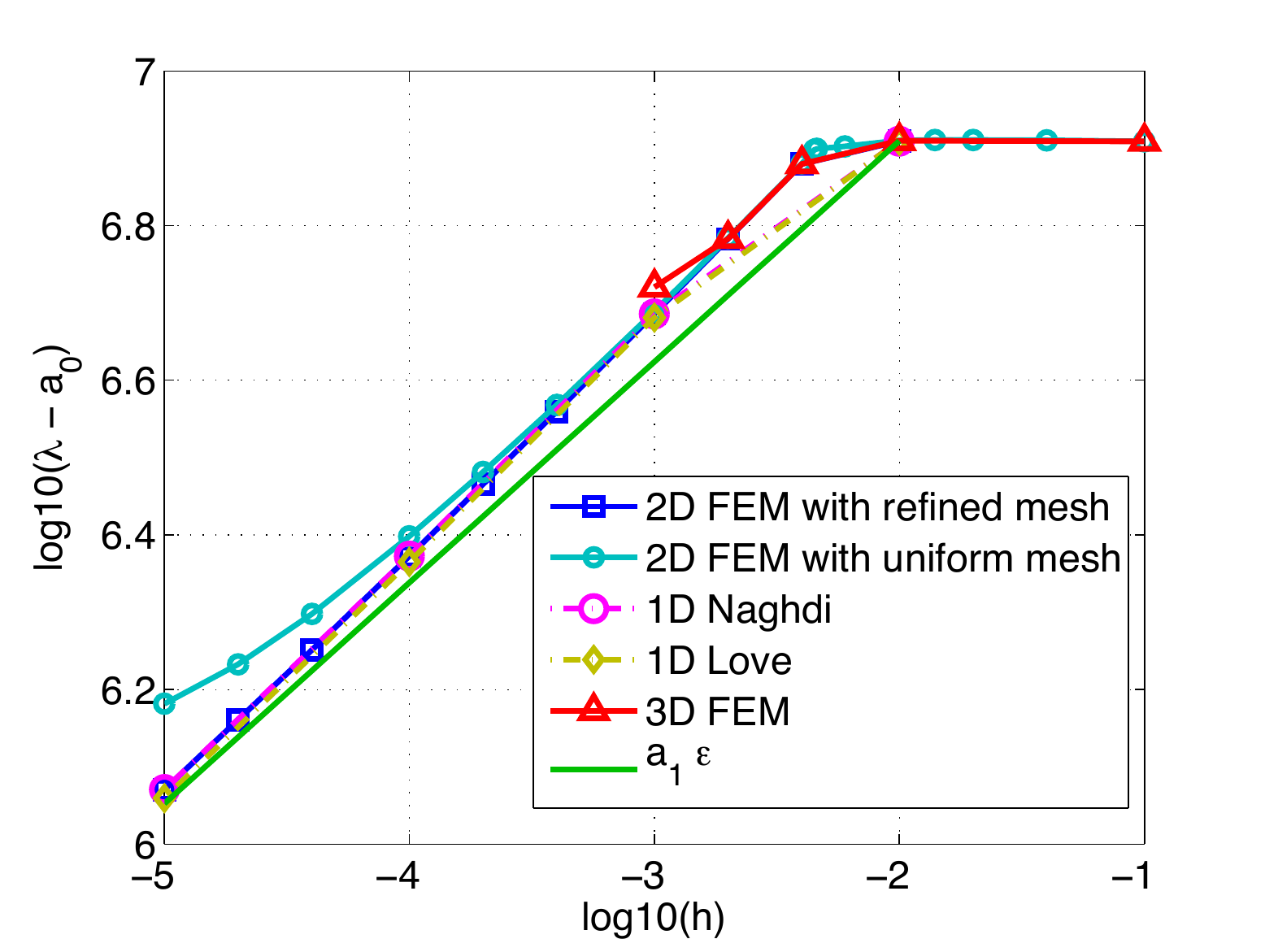}
   \caption{\small Shell \eqref{eq:para}-\eqref{eq:ell}: Computed values of $\lambda_\rL(\varepsilon)-\am_0$ versus  thickness $h=2\varepsilon$. 1D Naghdi and Love models in \cite{ArtioliBeiraoHakulaLovadina2009}. Asymptotics $h\mapsto \am_0+\am_1\,\varepsilon \simeq (0.1724+1.403\cdot \varepsilon)\,E/\rho$.}
\label{f:N6}\end{figure}

Here we present 2D computations with two different meshes. The uniform mesh has $2\times8$ curved elements of geometrical degree 3 (2 in the thickness direction, 8 in the meridian direction) and the interpolation degree is equal to 6. In the refined mesh, we add 8 points in the meridian direction, at distance $\cO(\varepsilon)$, $\cO(\varepsilon^{3/4})$, $\cO(\varepsilon^{1/2})$, and $\cO(\varepsilon^{1/4})$ from each lateral boundary, see figure \ref{f:N7}. So the mesh has the size $2\times16$. The geometrical degree is still 3 and the interpolation degree, 6. In this way we are able to capture these eigenmodes that concentrate at the scale $d/\varepsilon^{2/7}$, where $d$ is the distance to the lateral boundaries. In fact eigenmodes also contain terms at higher scales, namely $d/\varepsilon^{3/7}$ (membrane boundary layers), $d/\varepsilon^{1/2}$ (Koiter boundary layers), and $d/\varepsilon$ (3D plate boundary layers).
\begin{figure}[ht]
\includegraphics[scale=0.25]{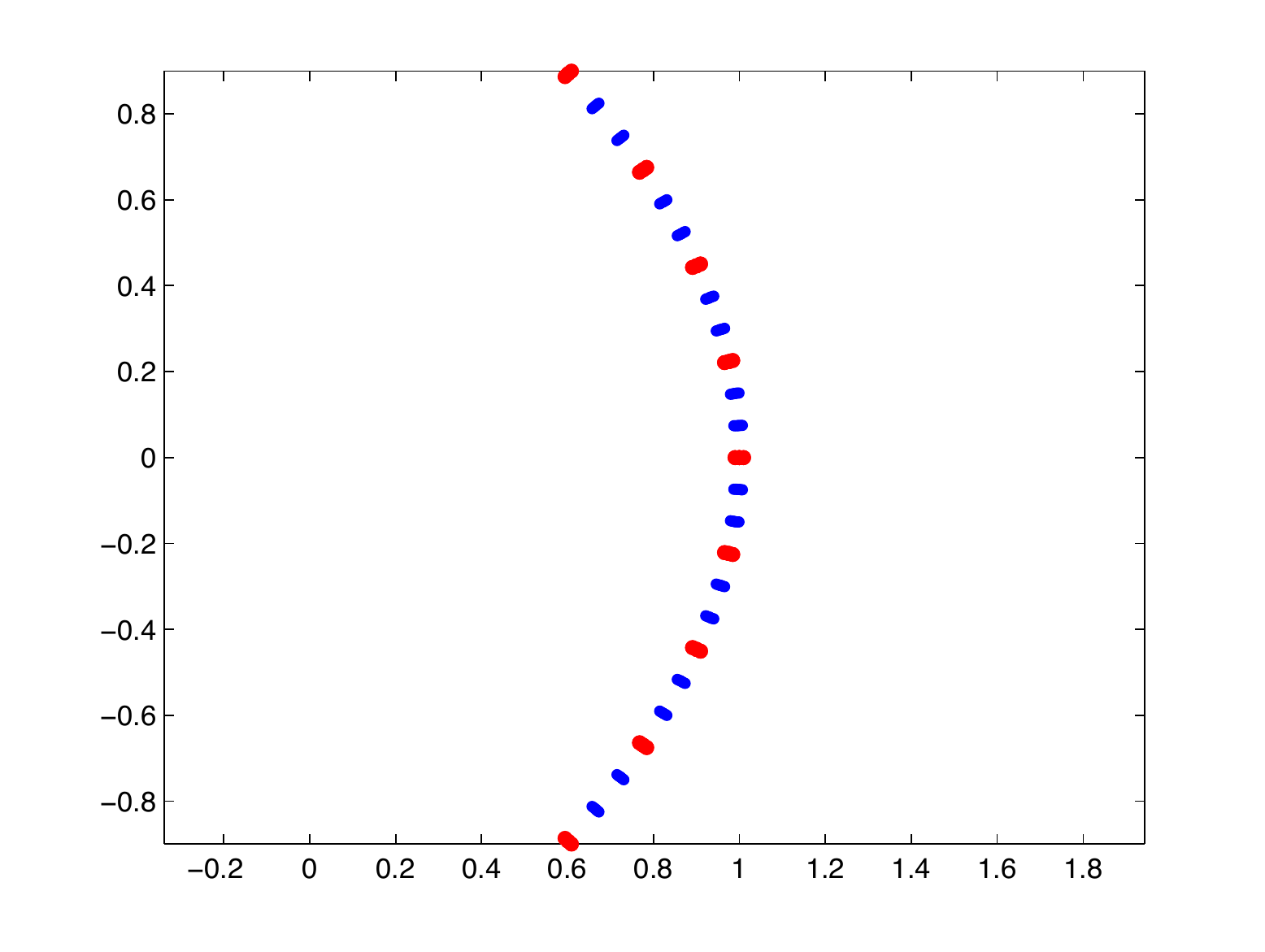}
\includegraphics[scale=0.25]{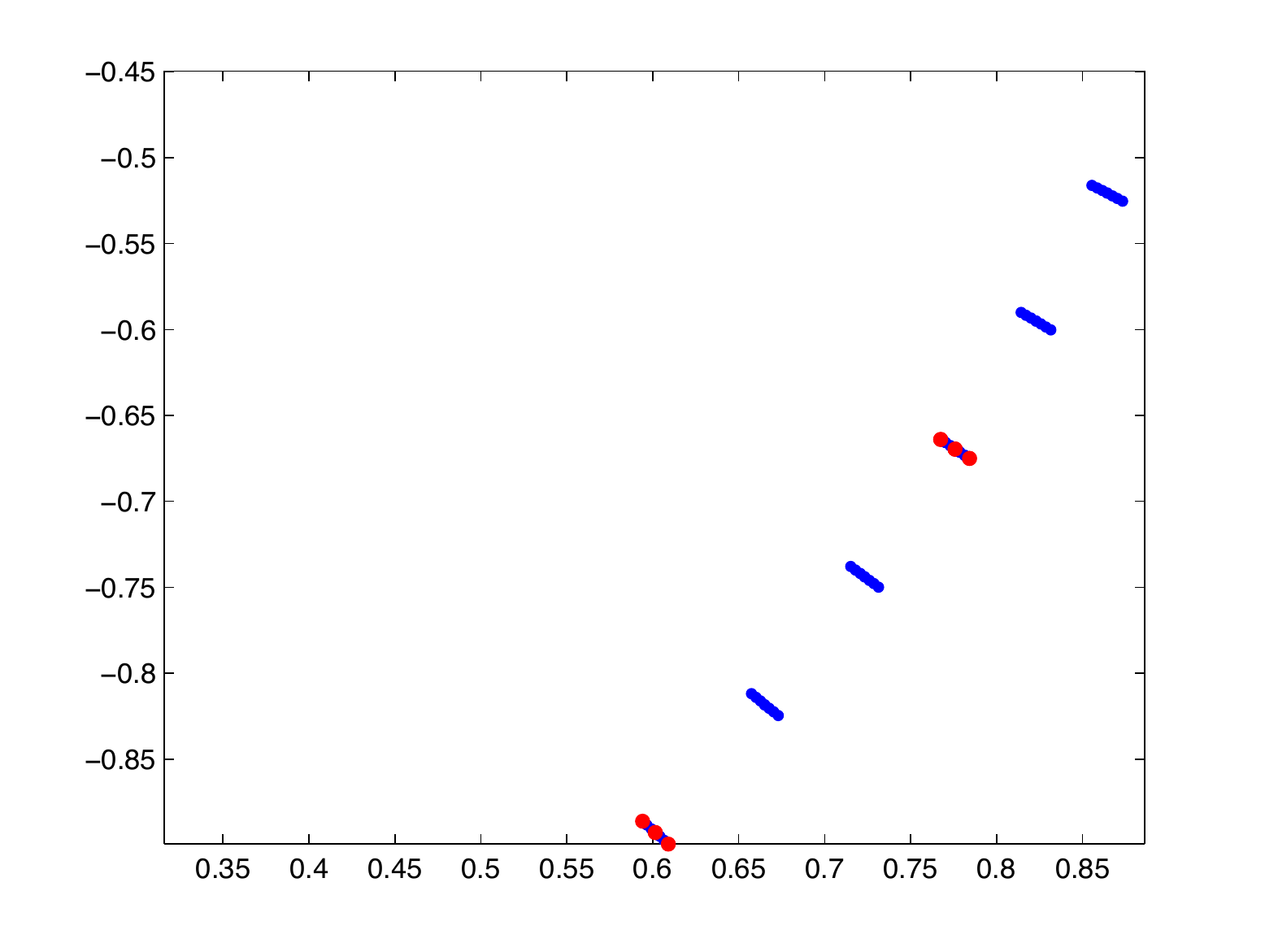}
\includegraphics[scale=0.25]{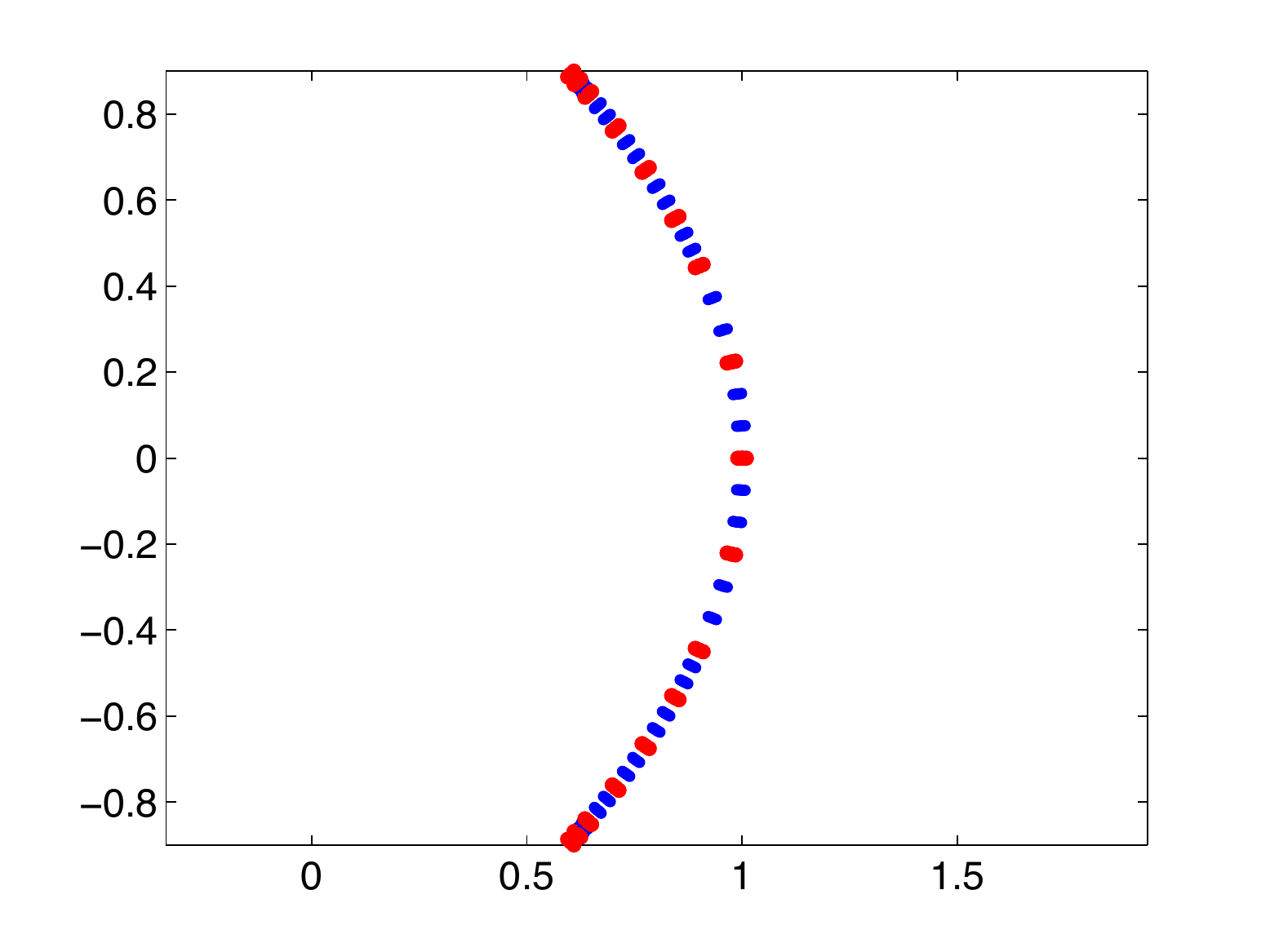}
\includegraphics[scale=0.25]{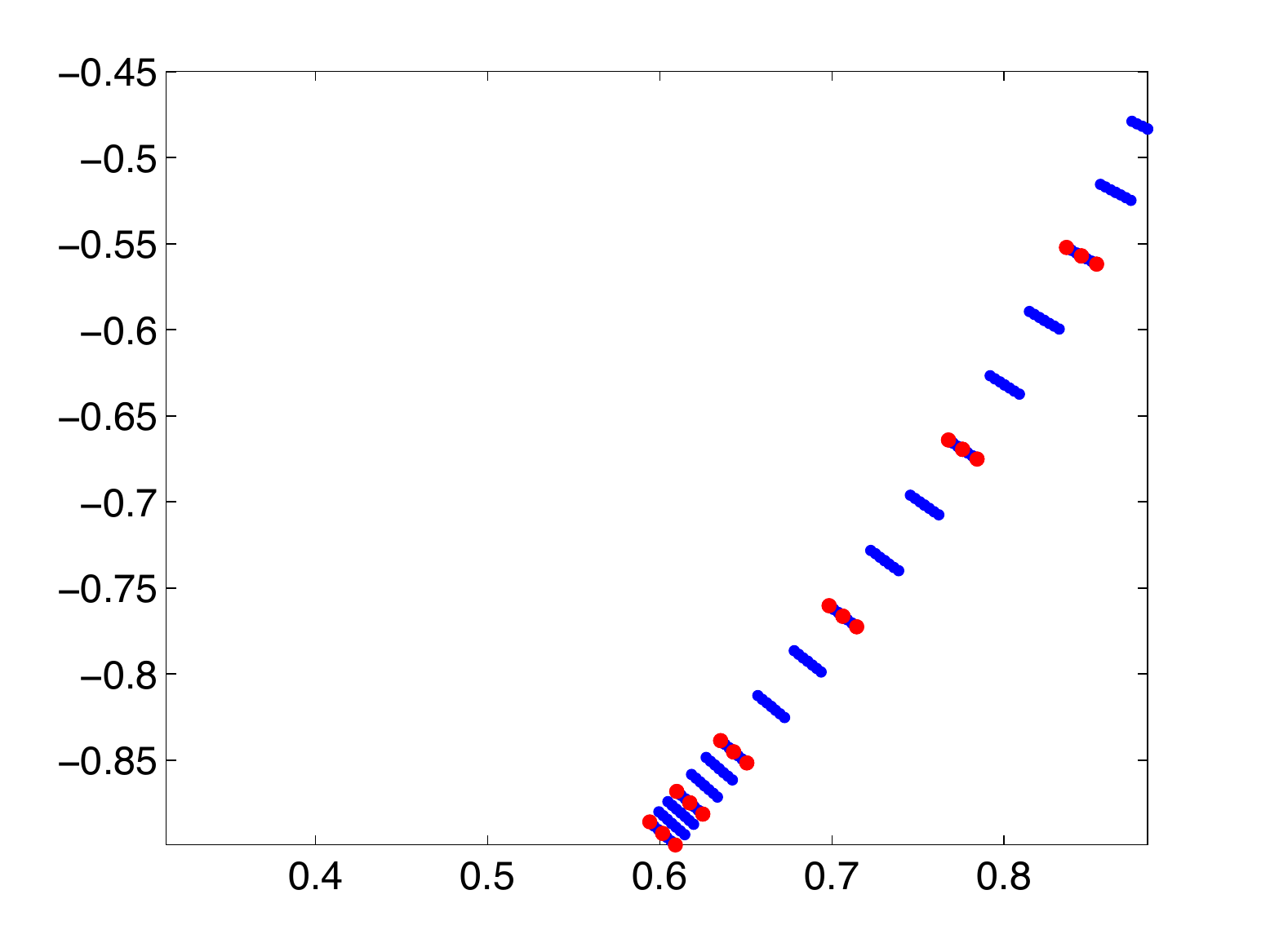}
   \caption{\small Meshes for shell \eqref{eq:para}-\eqref{eq:ell} and $\varepsilon=0.01$: Uniform (left), refined (right).}
\label{f:N7}\end{figure}

A further, more precise comparison of the five families of computations with the asymptotics is shown in figure \ref{f:N8} where the ordinates represent now $\log_{10}(\lambda-\am_0-\am_1\varepsilon^{2/7})$. These numerical results suggest that there is a further term in the asymptotics of the form $\am_2\varepsilon^{4/7}$. We observe a perfect match between the 1D Naghdi model and the 2D Lam\'e model using refined mesh. The Love-type model seems to be closer to the asymptotics. A reason could be the very construction of the asymptotics: They are built from a Koiter model from which we keep
\begin{itemize}
\item the membrane operator $\bM$,
\item the only term in $\partial^4_\varphi$ in the bending operator. Note that this term is common to the Love and Koiter models.  After angular Fourier transformation, the corresponding operator becomes $\rB_0(z)\,k^4$, with $\rB_0$ introduced in \eqref{eq:B0g}.
\end{itemize}

\begin{figure}[ht]
\includegraphics[scale=0.6]{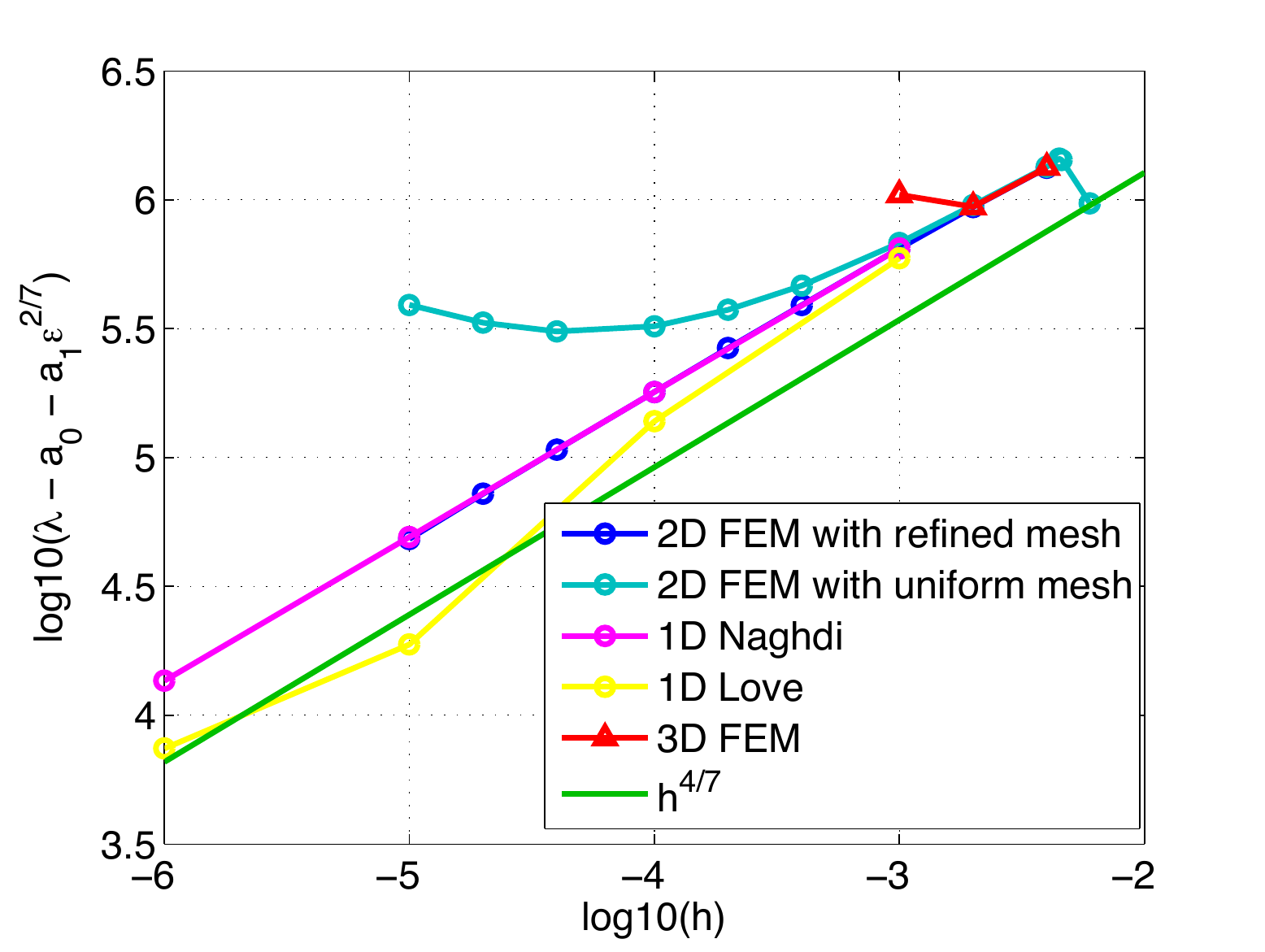}
   \caption{\small Shell \eqref{eq:para}-\eqref{eq:ell}: Computed values of $\lambda_\rL(\varepsilon)-\am_0-\am_1\varepsilon^{2/7}$ versus  thickness $h=2\varepsilon$. 1D Naghdi and Love models in \cite{ArtioliBeiraoHakulaLovadina2009}. The reference line is $h\mapsto \varepsilon^{4/7}\,E/\rho$.}
\label{f:N8}\end{figure}

The exponent $-\frac{3}{7}$ in \eqref{eq:asyell} is an exact fraction arising from an asymptotic analysis where the Airy equation $-\partial^2_ZU+ZU = \Lambda U$ on $\R_+$ with $U(0)=0$ shows up. Thus, the exponent $-\frac{2}{5}$
in \cite{ArtioliBeiraoHakulaLovadina2009} that is only an educated guess is probably incorrect.
We found in \cite[sect.\,6.3]{ChDaFaYo16a} this $-\frac{2}{5}$  exponent for another class of elliptic shells that we called {\em Gaussian barrels}, for which the function $\rH_0$ \eqref{eq:H0} attains its minimum inside the interval $\cI$ (instead of on the boundary for Airy barrels).

\section{Conclusion: The leading role of the membrane operator for the Lam\`e system}
\label{s:7}
We presented two families of shells for which the first eigenmode has progressively more oscillations as the thickness tends to $0$. The question is ``Can we predict such a behavior for other families of shells? What are the determining properties?''

In \cite{ChDaFaYo16a} we presented several more families of shells with same characteristics of the first mode. The common feature that controls such a behavior seems to be strongly associated to the membrane operator $\bM$. If we superpose to our dispersion curves $k\mapsto\lambda^{(k)}(\varepsilon)$ of the Lam\'e system the dispersion curves $k\mapsto\mu^{(k)}$ of the membrane operator, we observe convergence to the membrane eigenvalues as $\varepsilon\to0$ for each chosen value of $k$, see figure \ref{f:memb}. We also observe that for each chosen $\varepsilon$, the sequence $\lambda^{(k)}(\varepsilon)$ tends to $\infty$ as $k\to\infty$. The appearance of a global minimum of $\lambda^{(k)}(\varepsilon)$ for $k$ that tends to $\infty$ as $\varepsilon\to0$ occurs if the sequence $\mu^{(k)}$ has no global minimum: Its infimum is attained ``at infinity''.

\begin{figure}[ht]
\includegraphics[scale=0.5]{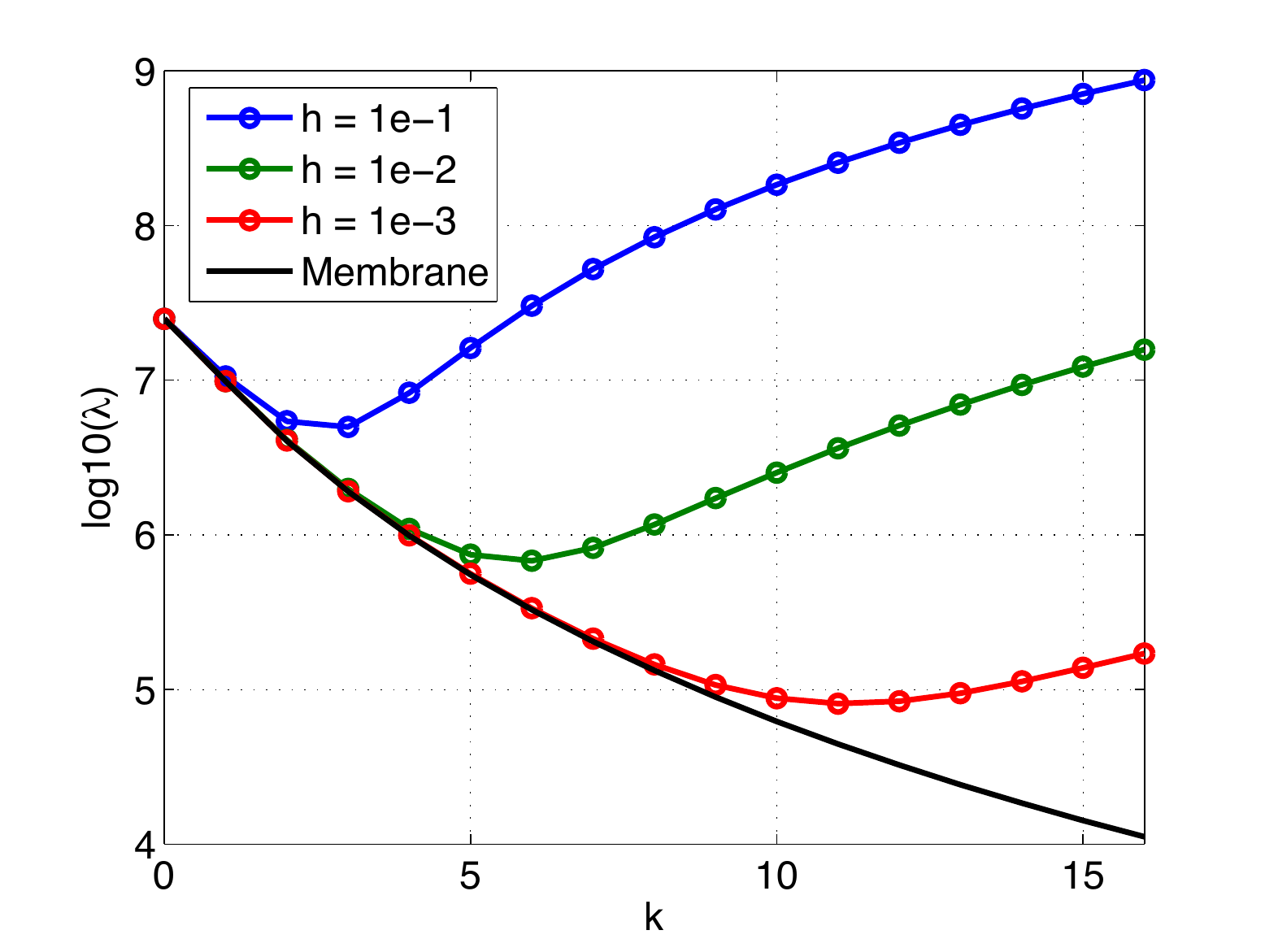}
\includegraphics[scale=0.5]{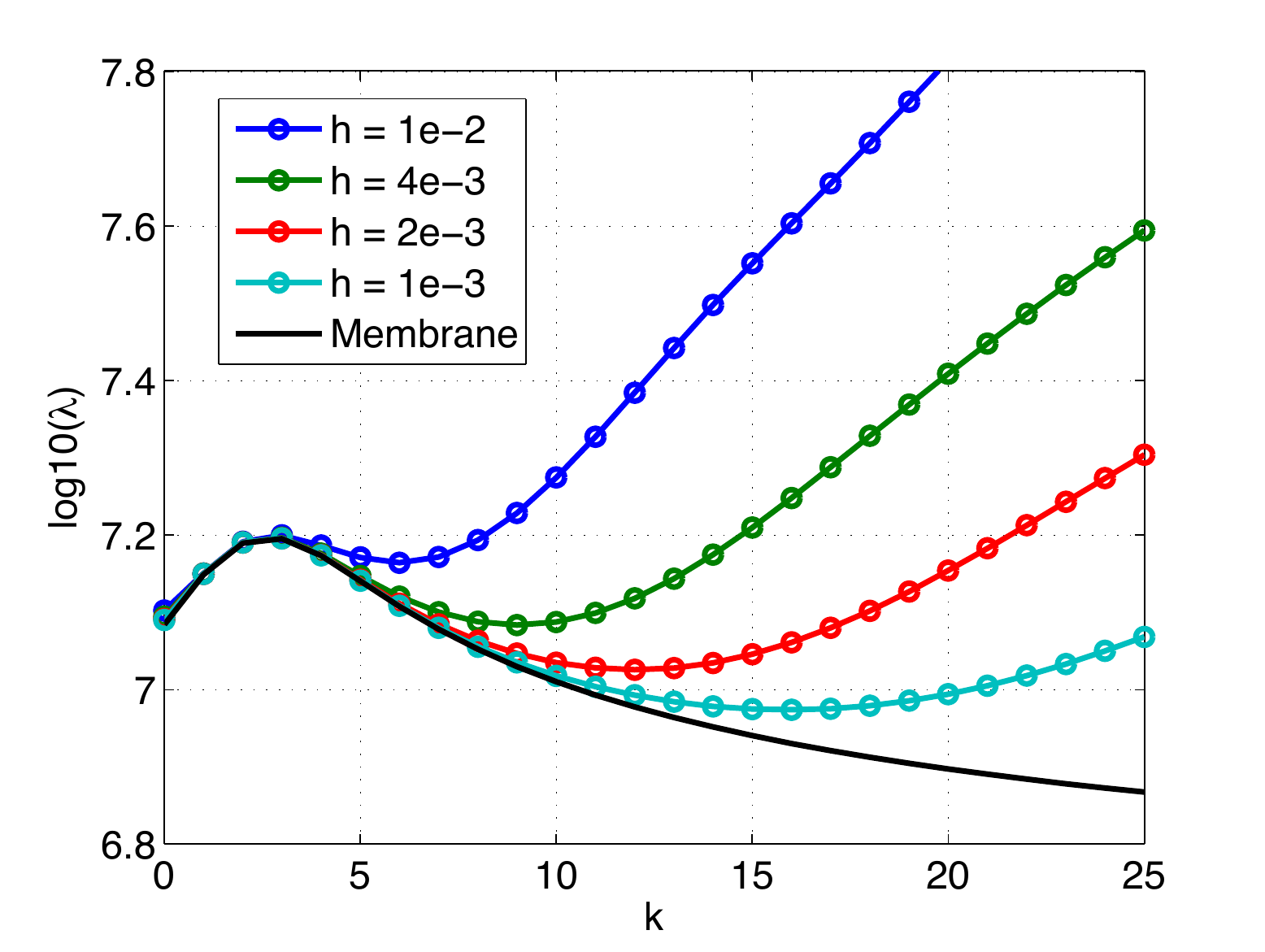}
   \caption{\small Cylinders (left) and Airy barrels (right): First eigenvalues of $\bL^{(k)}$ depending on $k$ for several values of thickness compared with first eigenvalues of $\bM^{(k)}$ (membrane).}
\label{f:memb}\end{figure}

For cylinders and cones, the sequence $\mu^{(k)}$ tends to $0$ as $k\to\infty$. Hence the sensitivity. For elliptic shells, the sequence $\mu^{(k)}$ tends to a limit that coincides with the minimum of the function $\rH_0$. Sensitivity depends on whether  $\mu^{(k)}$ has a minimum lower than this value. From our previous study it appears that any configuration is possible.
For hyperbolic shells,  $\mu^{(k)}$ tends to $0$ so sensitivity occurs, cf \cite{BeiraoLovadina2008,ArtioliBeiraoHakulaLovadina2008,ArtioliBeiraoHakulaLovadina2009} but the analysis of the coefficients in asymptotics cannot be performed by the method of \cite{ChDaFaYo16a}.

A natural question that comes to mind is: Are there other types of axisymmetric structures that behave similarly?
Rings (curved beams) are conceivable -- The recent work \cite{ForgitLemoineLemarrecRakoto16} tends to prove that sensitivity does not occur for thin rings with circular or square sections.
\bibliographystyle{siam}
\bibliography{shell}

\end{document}